\documentclass{lms}
\usepackage{pdflscape}
\usepackage{geometry}
\usepackage{latexsym}
\usepackage{indentfirst}
\usepackage{graphicx}
\usepackage{pdfpages}
\usepackage{amsmath}
\usepackage{amssymb}
\usepackage{bm}

\newtheorem{thm}{Theorem}[section]
\newtheorem{lem}[thm]{Lemma}
\newtheorem{prop}[thm]{Proposition}
\newtheorem{coro}[thm]{Corollary}

\newnumbered{defi}{Definition}
\newnumbered{rem}[thm]{Remark}
\newnumbered{ex}{Example}

\title[Irreducible recurrence and extremality]
 {Irreducible recurrence, ergodicity, and extremality of invariant measures for resolvents} 

\author{Lucian Beznea, Iulian C\^impean, and Michael R\"ockner}

\classno{60J35 (primary), 60J25, 60J40, 60J45, 31C25, 37A30, 37C40, 37L40, 82B10 (secondary)}

\extraline{Financial support by the Deutsche Forschungsgemeinschaft, project number RO1195/10-1 
is gratefully acknowledged.}

\begin{document}
\maketitle

\begin{abstract}
We analyze the transience, recurrence, and irreducibility properties 
of general sub-Markovian resolvents of kernels and their duals, with respect to a fixed sub-invariant measure $m$. 
We give a unifying characterization of the invariant functions, 
revealing the fact that an $L^p$-integrable function is harmonic if and only if it is harmonic with respect to the weak dual resolvent.
Our approach is based on potential theoretical techniques for resolvents in weak duality. 
We prove the equivalence between the $m$-irreducible recurrence of the resolvent  
and the extremality of $m$ in the set of all invariant measures, and we apply this result to the extremality of Gibbs states. 
We also show that our results can be applied to non-symmetric Dirichlet forms, in general and in concrete situations.
A second application is the extension of the so called {\it Fukushima ergodic theorem} for symmetric Dirichlet forms 
to the case of sub-Markovian resolvents of kernels.
\end{abstract}

\section{Introduction}

Questions on recurrence, transience and irreducibility of Markov processes were treated in various frames and with specific tools, both from probabilistic and analytic view point: see \cite{ChFu11}, 
\cite{Ge80}, \cite{Os92}, \cite{St94}, \cite{Fu07}, \cite{FuOsTa11}, and \cite{MaUeWa12} for continuous time processes, as well as \cite{MeTw93} and \cite{No97} for Markov chains, and the references therein.

The purpose of this paper is twofold: first, we aim to clarify the connection between different definitions for transience, recurrence, and irreducibility, and unify various characterizations of these notions. 
Second, we want to analyze whether transience, recurrence, and irreducibility are stable when passing to the dual structure, i.e. the dual Markov process or the dual resolvent, respectively, with the underlying measure being a sub-invariant measure for the initial resolvent.
On the way, we also obtain a number of new results on the subject, based on potential theoretical techniques.

Motivated by relevant examples arising mainly in infinite dimensional settings, 
we present here an approach to this subject in an $L^p$-context,
for sub-Markovian resolvents. It turns out to be a unifying method, 
in particular, revealing  applications to invariant and Gibbs measures.

The structure and main results of this paper are as follows. 

In the first part of Section 2 we study different characterizations of transience, recurrence, and irreducibility of a sub-Markovian resolvent of kernels $\mathcal{U}$ on a Lusin measurable space $E$ with respect to a $\sigma$-finite sub-invariant measure $m$. 
We emphasize that  we do not assume any continuity of the resolvent and our proofs rely on the weak duality for the resolvent $\mathcal{U}$, and corresponding potential theoretical techniques, see (\ref{eq 2.1}) below, which is in contrast to the ones in \cite{Fu07} and \cite{FuOsTa11}, where main ingredients are {\it Hopf's maximal inequality} and the continuity of the transition function. 
When $\mathcal{U}$ is the resolvent of a right process, we show that $m$-transience and $m$-irreducible recurrence are respectively equivalent with the transience and recurrence of the process in the stronger sense of \cite{Ge80}, outside some $m${\it -inessential} set; see Propositions \ref{prop 2.1.14} and \ref{prop 2.1.18}.
This probabilistic counterpart was  studied in \cite{FuOsTa11} for $m$-symmetric Hunt processes.
Then, we give a characterization for invariant functions in $L^p(E, m)$, $1 \leq p \leq \infty$, unifying the approaches from stochastic processes, Dirichlet forms, positivity preserving semigroups, and ergodic theory (see Theorem \ref{thm 2.19}). 
Our results also cover and extend the ones in \cite{Sc04}, 
and we shall use it in Sections 4 and 5 to prove the equivalence of irreducibility and extremality 
of invariant measures, resp. extremality of Gibbs states. 
A second consequence of Theorem \ref{thm 2.19} states that an element $u$ from the kernel of the generator of an $L^p$-strongly continuous sub-Markovian resolvent of contractive operators also belongs to the kernel of the co-generator on $L^p$ induced by weak duality; see Corollary \ref{prop 3} and the discussion at the beginning of Section 2.
In Section 3 we apply the results of the previous one to prove the equivalence of irreducible recurrence and ergodicity, as stated in Proposition \ref{prop 3.5}, of a sub-Markovian resolvent of kernels with respect to a sub-invariant $\sigma$-finite measure, extending the so called {\it Fukushima ergodic theorem} for a (quasi)regular Dirichlet form; see \cite{FuOsTa11}, Theorem 4.7.3 and \cite{AlKoRo97a}, Theorem 4.6.
The key ingredient is Theorem \ref{thm 3.1} which states the strong convergence of an $L^p$-uniformly bounded resolvent family of continuous operators $(\alpha U_\alpha)_{\alpha > 0}$ to the projection on the kernel of $\mathcal{I} - \beta U_\beta$, as $\alpha$ tends to $0$, for one (hence for all) $\beta > 0$. 

The central result of Section 4 is Theorem \ref{thm 2.26} which states that the sub-Markovian resolvent of kernels $\mathcal{U}$ 
is $m$-recurrent and $m$-irreducible if and only if the measure $m$ is extremal in the set of all invariant probability measures for $\mathcal{U}$. This extends results from \cite{AlKoRo97a}, \cite{AlKoRo97b}, and \cite{DaZa96}, Section 3.1, concerning  the ergodicity  and extremality of invariant measures.

In Section 5 we apply the obtained results on irreducibility and extremality of invariant measures to the context of (non-symmetric) Dirichlet forms.
In Corollary \ref{coro 5.3} we give a characterization for the irreducibility of a  Dirichlet form. 
It improves the one in \cite{AlKoRo97a}, Proposition 2.3, where the forms are symmetric, recurrent, and given by a square field operator. 
We would like to point out another consequence of Corollary \ref{coro 5.3}, namely that both the recurrence and the irreducibility of a strongly sectorial (non-symmetric) Dirichlet form is equivalent to the respective property of its symmetric part. We illustrate this by a concrete example in infinite dimensions (see Corollary \ref{coro 5.8}).

The main results of this last section are given in a subsection on the extremality of Gibbs states. Recall that  in \cite{AlKoRo97a} the authors 
extend  classical results  of Holley and Stroock for the  Ising model, proving that a Gibbs state is extremal if and only if the corresponding Dirichlet form is irreducible (or equivalently ergodic), for classes of lattice models with non-compact, but linear spin space. 
In particular, numerous examples of irreducible Dirichlet forms on infinite dimensional state space are obtained. 
For applications to more general models see \cite{AlKoRo97b}. 
Our purpose is to recapture two of the main results in \cite{AlKoRo97a} as particular cases of Theorem \ref{thm 2.26} and thus to place the problem in a broader context. 
The key point is Theorem \ref{thm 5.5}, according to which the space of Gibbs measures which are absolutely continuous with respect to a fixed Gibbs measure $m$ coincide with the space of all $\mathcal{U}$-invariant probability measures which are absolutely continuous with respect to $m$; here, $\mathcal{U}$ is the resolvent of the Dirichlet form. 
Theorem \ref{thm 5.6} is the main result on the equivalence between the extremality of Gibbs states and irreducibility of the corresponding Dirichlet form.

In order to give a better overview of the paper, we summarize its structure and the main results in the following two diagrams.

\newgeometry{margin=0.3 cm}

\begin{landscape}

\begin{center}

\includegraphics[scale=0.8]{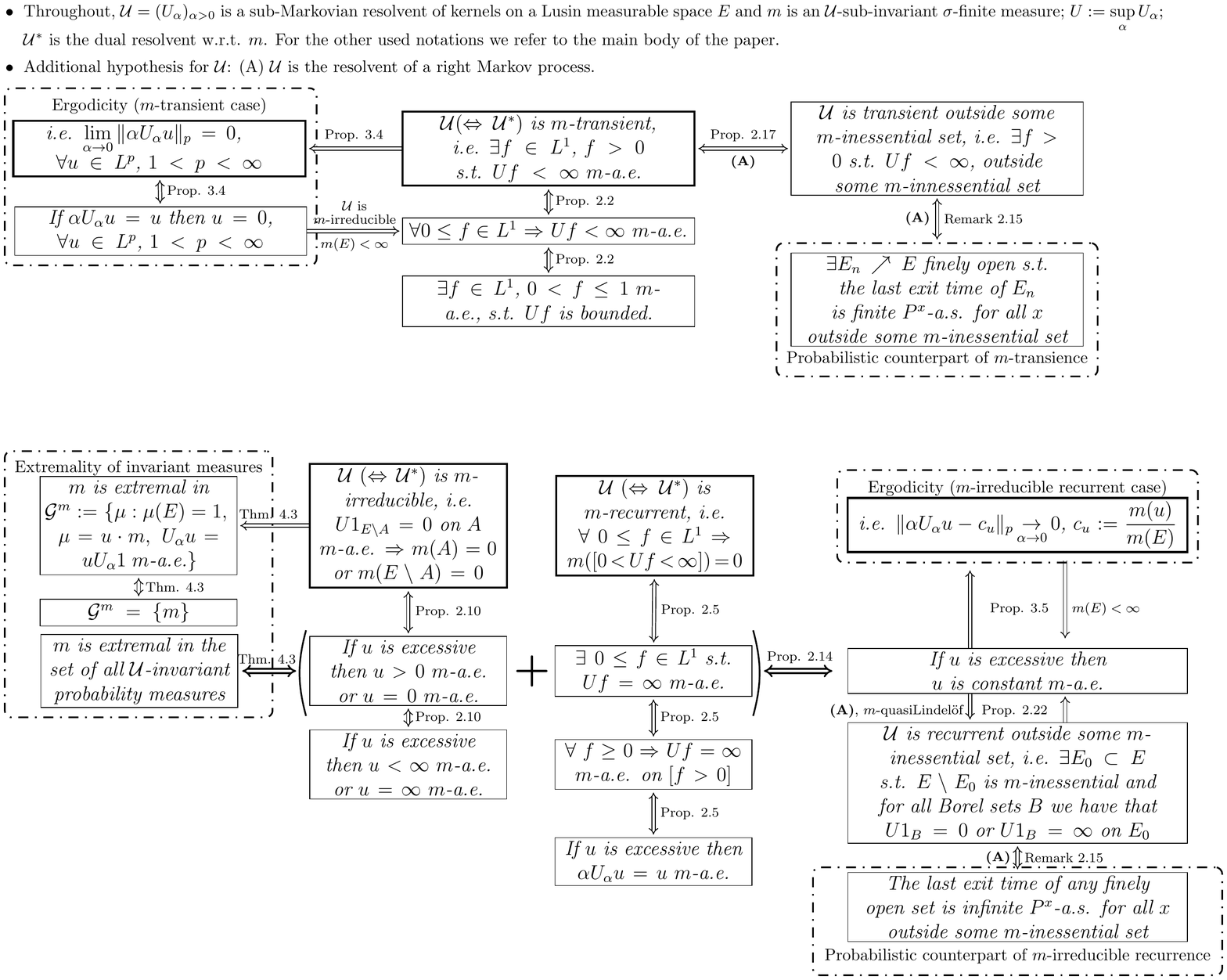}

\end{center}

\end{landscape}

\restoregeometry

Finally, we point out that the situation where we are given a sub-Markovian resolvent of contractive operators $\mathcal{V}=(V_\alpha)_{\alpha > 0}$ on $L^p(E, m)$, $p \in [1, \infty)$, where $m$ is a sub-invariant measure on $E$, is covered by our framework, since one can always construct a sub-Markovian resolvent of kernels $\mathcal{U}= (U_\alpha)_{\alpha > 0}$ on $(E, \mathcal{B})$ such that $U_\alpha = V_\alpha$, as operators on $L^p(E,m)$ for all $\alpha > 0$. We give more details on this in the beginning of the next section.

\section{Transience, recurrence, and irreducibility of a sub-Markovian resolvent of kernels}

\noindent
{\bf Preliminaries on resolvents of kernels, $L^p$-resolvents, and duality.}
Throughout we follow the terminology of \cite{BeBo04}. Let $(E, \mathcal{B})$ 
be a Lusin measurable space and $\mathcal{U} = (U_\alpha)_{\alpha > 0}$ be a sub-Markovian resolvent of kernels on $(E, \mathcal{B})$.  
Throughout, we denote by $p\mathcal{B}$ the space of all positive $\mathcal{B}$-measurable functions defined on $E$.
The {\it initial kernel} of $\mathcal{U}$ is defined as $U := \sup\limits_{\alpha \geq 0} U_\alpha = \sup\limits_{n} U_{\frac{1}{n}}$.  
Recall that a function $u \in p\mathcal{B}$ is called {\it $\mathcal{U}$-supermedian} if $\alpha U_{\alpha} u \leq u$ for all $\alpha > 0$; 
it is called {\it $\mathcal{U}$-excessive} if it is $\mathcal{U}$-supermedian and $\alpha U_{\alpha} u \nearrow u$ as $\alpha$ tends to infinity. 
We  denote by $\mathcal{S}(\mathcal{U})$ and $\mathcal{E}(\mathcal{U})$ the sets of all $\mathcal{U}$-supermedian (resp. $\mathcal{U}$-excessive) functions. 
A $\sigma$-finite measure $\mu$ is called {\it sub-invariant}  w.r.t.  $\mathcal{U}$ if $\mu \circ \alpha U_{\alpha} \leq \mu$, $\alpha > 0$.
For $\beta > 0$ we denote by $\mathcal{U}_{\beta}$ the $\beta$-order sub-Markovian resolvent of kernels associated to $\mathcal{U}$, that is $\mathcal{U}_{\beta} := (U_{\beta + \alpha})_{\alpha > 0}$.

If $m$ is a $\sigma$-finite sub-invariant measure, then by \cite{BeBo04}, Theorem 1.4.14, 
there exists a second sub-Markovian resolvent of kernels $\mathcal{U}^\ast = (U^\ast_\alpha)_{\alpha > 0}$ on $(E, \mathcal{B})$ such that 
\begin{equation} \label{eq 2.1}
\mathop{\int}\limits_{E} f U_{\alpha}g dm = \mathop{\int}\limits_{E} g U_{\alpha}^{\ast} f dm \; \mbox{for all} \; f,g \in p\mathcal{B} \; {\rm and} \; \alpha > 0.
\end{equation}
Such a sub-Markovian resolvent is uniquely determined $m$-a.e. and it is called the {\it adjoint} of $\mathcal{U}$ w.r.t. $m$. 

Using H\"older inequality and extending by linearity one can easily check that $\mathcal{U}$ 
becomes a sub-Markovian family of contractive operators on $L^p(E,m)$ for all $1 \leq p < \infty$. 
Conversely, if $\mathcal{V}:=(V_{\alpha})_{\alpha > 0}$ is a sub-Markovian 
resolvent of contractive operators on $L^p(E, m)$ for some $p \in [1, \infty)$, 
where $m$ is a $\sigma$-finite measure on $(E, \mathcal{B})$ such that $\int \alpha V_{\alpha}f dm \leq \int f dm$ for all $\alpha > 0$ and $f \in p\mathcal{B} \cap L^p(E, m)$, then 
by \cite{BeBo04}, Proposition 1.4.13 and Lemma A.1.9, 
there exist two sub-Markovian resolvents of kernels $\mathcal{U} = (U_{\alpha})_{\alpha >0}$ and 
$\mathcal{U}^{\ast} = (U_{\alpha}^{\ast})_{\alpha > 0}$ on $(E, \mathcal{B})$ such that:

$a)$ $U_{\alpha}=V_{\alpha}$ as operators on $L^p(E,m)$ for all $\alpha>0$;

$b)$ \; $\mathcal{U}$ and $\mathcal{U}^{\ast}$ are in {\it weak} duality with respect to $m$, that is (\ref{eq 2.1}) holds.

\noindent
Moreover, if $\mathcal{V}$ is strongly continuous, i.e. $\lim\limits_{\alpha \to \infty} \| \alpha U_\alpha f - f\|_{L^p}=0$ all $f \in L^p(E,m)$, 
$1 \leq p < \infty$, then by \cite{BeBoRo06a}, Remark 2.3, and Corollary 2.4, we also have

$c)$ \; $1 \in \mathcal{E}(\mathcal{U}_{\beta})\cap\mathcal{E}(\mathcal{U}^{\ast}_{\beta})$, 
$\sigma(\mathcal{E}(\mathcal{U}_{\beta})) = \mathcal{B} = \sigma(\mathcal{E}(\mathcal{U}^{\ast}_{\beta}))$;

$d)$ \; Every point of $E$ is a non-branch point for $\mathcal{U}$ and $\mathcal{U}^{\ast}$.\\

We note that, as in \cite{MaRo92}, Chapter II, Proposition 4.3, the strong continuity of $\mathcal{V}$ for $1 \leq p < \infty$ 
is satisfied if one can find a dense subset $D \subset L^p(E,m)$ such that $\alpha V_\alpha f \mathop{\longrightarrow}\limits_{\alpha \to \infty} f$ in $m$-measure for all $f \in D$.
A second approach to strongly continuity is given by the next result, 
for which we refer to \cite{BeBoRo06a}, Remark 2.3, and Corollary 2.4 and \cite{BeBo04}, Subsection 7.5.

\begin{prop} \label{prop 2.0} 
Let $\mathcal{U}$ be a sub-Markovian resolvent of kernels on $(E, \mathcal{B})$. 
If $\mathcal{E}(\mathcal{U}_\beta)$ is min-stable, contains the positive constant functions and generates $\mathcal{B}$ then $\mathcal{U}$ becomes a strongly continuous sub-Markovian resolvent of contractive operators on $L^p(E, m)$ for every $\sigma$-finite sub-invariant measure $m$ and $1 \leq p < \infty$.
\end{prop}

We would also like to stress out that if we deal with a strongly continuous sub-Markovian resolvent of contractions $\mathcal{U}$ on $L^p(E,m)$ then one can always find a larger Lusin topological space $E \subset E_1$, $E \in \mathcal{B}(E_1)$, $\mathcal{B} = \mathcal{B}(E_1)|_E$, 
and an $E_1$-valued right Markov process such that its resolvent $\mathcal{U}^1$ regarded on $L^p(E,\overline{m})$, 
coincides with $\mathcal{U}$ and $U^1_\alpha(1_{E_1 \setminus E}) = 0$, 
where $\overline{m}$ is the measure on $(E_1, \mathcal{B}(E_1))$ extending $m$ by zero on $E_1 \setminus E$; see \cite{BeBoRo06a}, Theorem 2.2, 
and also \cite{BeRo15} for the extension to $E_1$ of the adjoint  resolvent.
Taking into account that the properties of $m$-transience, $m$-recurrence, and $m$-irreducibility are preserved by modifying the initial space with some zero measure set, the results presented in this paper have a probabilistic counterpart.

As mentioned in Introduction, due to the above remarks 
we can assume without loss of generality that $\mathcal{U}$ is a sub-Markovian resolvent of kernels on $(E, \mathcal{B})$. 
We also fix a $\sigma$-finite sub-invariant measure $m$.

\begin{defi*}[(cf. \cite{Fu07})] 
i) The resolvent $\mathcal{U}$ is called {\it $m$-transient} provided there exists $f \in p\mathcal{B} \cap L^1(E, m)$ such that $m$-a.e. we have $f >0$ and $U f < \infty$.

\vspace{0.2cm}

ii) The resolvent $\mathcal{U}$ is called {\it $m$-recurrent} if for every $f \in p\mathcal{B}\cap L^1(E, m)$ we have $m(\{ x \in E : 0 < Uf(x) < \infty \}) = 0$.
\end{defi*}

\begin{prop} \label{prop2.1} 
The following assertions are equivalent

i) $\mathcal{U}$ is $m$-transient.

ii) $\mathcal{U}^{\ast}$ is $m$-transient.

iii) There exists $f_0 \in p\mathcal{B} \cap L^1(E, m)$, $0 < f_0 \leq 1$ $m$-a.e., such that $U f_0$ is bounded.

iv) For every $f \in p\mathcal{B} \cap L^1(E, m)$ we have $U f < \infty$ $m$-a.e.
\end{prop}
\begin{proof}
i) $\Rightarrow$ iii) Let $f \in p\mathcal{B} \cap L^1(E, m)$ be such that $m$-a.e. $f > 0$ and $U f < \infty$. 
For every $n \in \mathbb{N}^{\ast}$ let us put
$A_n := \{ x \in E : Uf(x) \leq n \}.$
Then clearly $m(E \setminus \mathop{\bigcup}\limits_{n} A_n) = 0$ and by the complete maximum principle we get
$U(f\cdot 1_{A_n}) \leq n.$
If we put
$f_0 := \inf (1, \mathop{\sum}\limits_{n = 1}^{\infty} \dfrac{1}{n\cdot 2^n} f \cdot 1_{A_n})
$
then $0 < f_0 \leq 1$ $m$-a.e. and $U f_0 \leq 1$ on $E$.

ii) $\Rightarrow$ iv) 
Applying i) $\Rightarrow$ iii) for $\mathcal{U}^{\ast}$ we get a function $g_0 \in p\mathcal{B} \cap L^1(E, m)$, $0 < g_0 \leq 1$ $m$-a.e. such that $U^{\ast}g_0$ is bounded.
If $f \in p\mathcal{B} \cap L^1(E, m)$ then $\int g_0 U f dm = \int fU^{\ast}g_0 dm < \infty$. Since $g_0 >0$ $m$-a.e. we deduce that $U f < \infty$ $m$-a.e.

The implication iv) $\Rightarrow$ i) is trivial and therefore we have i) $\Leftrightarrow$ ii).
\end{proof}

\begin{rem}  
a) For a different proof of equivalence i) $\Leftrightarrow$ iv), which makes heavy use of Hopf's maximal inequality and the continuity of the transition function, see \cite{Fu07}, Proposition 1.1, i) and \cite{FuOsTa11}, Lemma 1.5.1. 

b) If $\mathcal{U}$ is the resolvent of a right process and $m$ is a $\sigma$-finite sub-invariant measure, then by \cite{Fu07}, Section 3.1, $\mathcal{U}$ is $m$-transient if and only if there exists a sequence of Borel finely open sets $(B_n)_{n \geq 1}$ increasing to $E$ such that q.e. in $x \in E$ (i.e. for all $x \in E$ outside some $m$-polar set) the last exit time of $B_n$ is finite $P_x$-a.e. In particular, if q.e. in $x \in E$ the process has finite lifetime $P_x$-a.e., then $\mathcal{U}$ is $m$-transient.
\end{rem}

For the reader convenience we present the proof of the following essentially known result (see e.g. \cite{BeBo04}, Proposition 1.1.11).

\begin{prop} \label{prop 2.2} 
Let $v \in p\mathcal{B}$ be such that
$\alpha U_{\alpha} v \leq v \; m\mbox{-a.e. for all} \; \; \; \alpha > 0.$
Then there exists an $\mathcal{U}$-excessive function $v'$ such that
\[
v = v' \; m\mbox{-a.e.}
\]
Moreover, if $v$ is bounded then $v'$ can be chosen bounded. 
\end{prop}
\begin{proof}
We consider the sequence $(v_n)_{n \geq 0}$ defined inductively as follows
\[
v_0 := v \; {\rm and} \; v_{n+1} := \mathop{\sup}\limits_{\alpha \in \mathbb{Q}_{+}}\sup(v_n, \alpha U_{\alpha} v_n).
\]
Then clearly the sequence is increasing and for all $\alpha \in \mathbb{Q}_{+}$ and $n \in \mathbb{N}$ we have $\alpha U_{\alpha} v_n \leq v_{n+1}$. 
Taking $v'' := \sup v_n$ we get that $\alpha U_{\alpha} v'' \leq v''$ for all $\alpha \in \mathbb{Q}_{+}$ and therefore $v''$ is an $\mathcal{U}$-supermedian function. 
Because $v_n = v$ $m$-a.e. for all $n$ we obtain that $v = v''$ $m$-a.e. 
The required function $v'$ will be the $\mathcal{U}$-excessive regularization of $v''$, $v' = \mathop{\sup}\limits_{k} k U_{k} v''$. 
In order to prove the second part of the statement we only have to notice that if $v$ is bounded then so are $v''$ and its $\mathcal{U}$-excessive regularization $v'$.
\end{proof}

The next proposition collects several characterizations of $m$-recurrence.

\begin{prop} \label{prop2.3} 
The following assertions are equivalent.

i) $\mathcal{U}$ is $m$-recurrent.

ii) $\mathcal{U}^{\ast}$ is $m$-recurrent.

iii) There exists $f_0 \in p\mathcal{B} \cap L^1(E, m)$ such that $U f_0 = +\infty$ $m$-a.e.

iv) For every $f \in p\mathcal{B}$ we have
\[
U f = +\infty \quad m\mbox{-a.e. on} \; \; \; [f > 0].
\] 

v) For every $v \in L^\infty(E, m)$ such that $\alpha U_\alpha v \leq v$ $m$-a.e. for all $\alpha > 0$ we have
\[
\alpha U_{\alpha} v = v \quad m\mbox{-a.e. for all} \; \; \; \alpha > 0.
\]

v$'$) For every $\mathcal{U}$-excessive function $v$ we have
\[
\alpha U_{\alpha} v = v \quad m\mbox{-a.e. for all} \; \; \; \alpha > 0.
\]
\end{prop}
\begin{proof}
i) $\Rightarrow$ iv). Let $f \in p\mathcal{B} \cap L^1(E, m)$ and $A = [U f = 0]$.
It follows that $U_{\alpha}(f 1_A) = 0$ and thus $m([f > 0] \cap A) = 0$.
By hypothesis we conclude that $U f = +\infty$ $m$-a.e. on $[f > 0]$.

iv) $\Rightarrow$ i) Let $M = [0 < U f < \infty]$. 
By iv) we get $[U f = \infty] \supset [f > 0]$ $m$-a.e. 
Let $B = [U f < \infty] \cap [f > 0]$. 
Then $m(B) = 0$ and $[f 1_{E \setminus B} > 0] \subset [U f = \infty]$. 
It follows that $U(f 1_{E \setminus B}) \leq \dfrac{1}{n}U f$ for all $n$ and thus 
$U(f 1_{E \setminus B}) = 0$ on $[U f < \infty]$, hence $U(f) = 0$ $m$-a.e. on $[U f < \infty]$, i.e. $m(M) = 0$.

Clearly we have iv) $\Rightarrow$ iii)

iii) $\Rightarrow$ ii) Let $f_0 \in p\mathcal{B} \cap L^1(E, m)$ be such that $U f_0 = +\infty$ $m$-a.e.
If $g \in p\mathcal{B}\cap L^1(E, m)$ is such that $\int g dm > 0$ then we claim that 
for every $A \in \mathcal{B}$, $A \subset [g > 0]$ such that $U^{\ast}(g 1_A)$ is bounded we have $m(A) = 0$.
Indeed, in the contrary case we have $m(A) > 0$ and since $U^{\ast}(g 1_A)$ is bounded we arrive at the contradictory relation
\[
\infty = \int g 1_A U f_0 dm = \int f_0 U^{\ast}(g 1_A) dm < \infty.
\]
We conclude that $U^{\ast} g = \infty$ $m$-a.e. on $[g > 0]$. 
By the implication iv) $\Rightarrow$ i) applied to $\mathcal{U}^{\ast}$ we deduce that assertion ii) holds.

ii) $\Rightarrow$ iv). Let $0<g<1$ such that $m(g) < \infty$ and take $f \in p\mathcal{B}$.
Then $\inf(f,g) \in p\mathcal{B} \cap L^1(E, m)$, $U\inf(f,g) \leq Uf$ and $[\inf(f,g)>0] = [f>0]$.
Therefore, we may assume that $f \in p\mathcal{B} \cap L^1(E, m)$.
Now suppose that there exists $\mathcal{B} \ni A \subset [f > 0]$ with $m(A) > 0$ and $U f$ bounded on $A$. 
It follows that the function $U(f 1_A)$ is bounded and
\[
\int f 1_A U^{\ast}(f 1_A) dm = \int f 1_A U(f 1_A) dm < \infty.
\]
Consequently $U^{\ast}(f 1_A) < \infty$ $m$-a.e. on $A$ and by hypothesis ii) we get that $U^{\ast}(f 1_A) = 0$ $m$-a.e. on $A$. 
We deduce that $\int f 1_A U(f 1_A) dm = 0$, hence $U(f 1_A) = 0$ $m$-a.e. on $A$, which leads to the contradictory relation $m(A) = 0$. 
We conclude that $U f = +\infty$ $m$-a.e. on $[f > 0]$.

iv) $\Rightarrow$ v) Let $\alpha, \beta > 0$ and $v \in L^\infty(E, m)$ such that $\alpha U_\alpha v \leq v$ $m$-a.e.
Then we have
\[
U_{\beta}(v - \alpha U_{\alpha} v) = U_{\alpha}(v - \beta U_{\beta} v) \leq \dfrac{\| v - \beta U_{\beta} v \|_\infty}{\alpha} \leq \dfrac{2 \|v\|_\infty}{\alpha} \; \; \; m\mbox{-a.e.} 
\]
and so
\[
U(v - \alpha U_{\alpha} v) \leq \dfrac{2 \|v\|_\infty}{\alpha} \quad m\mbox{-a.e.}
\]
If $m(v - \alpha U_{\alpha} v) > 0$ then by hypothesis iv) we get the contradictory relation
$$
\dfrac{2 \|v\|_\infty}{\alpha} \geq U(v - \alpha U_{\alpha} v) = +\infty \quad m\mbox{-a.e. on} \; [v - \alpha U_{\alpha} v > 0].
$$
Hence $\alpha U_{\alpha} v = v$ $m$-a.e.

v) $\Rightarrow$ v$'$) If $\alpha > 0 $, $v \in \mathcal{E}(\mathcal{U})$ and $v_n := \inf(v, n)$, $n \in \mathbb{N}^{\ast}$, 
then $(v_n)_n \subset b\mathcal{E}(\mathcal{U})$,  $\alpha U_{\alpha} v_n = v_n$ $m$-a.e. and $v_n \nearrow v$ pointwise. 
Hence $\alpha U_{\alpha} v = v$ $m$-a.e.

\vspace{0.2cm}

v$'$) $\Rightarrow$ i) If $f \in p\mathcal{B} \cap L^1(E, m)$ and $\alpha > 0$ then 
$U f = U_{\alpha} f + \alpha U_{\alpha} U f$ and from v$'$) we have $\alpha U_{\alpha} U f = U f$ $m$-a.e. 
It follows that for all $\alpha > 0$ we have $m$-a.e. $U_{\alpha} f = 0$ on $[U f < \infty]$ and we conclude that $U f = 0$ $m$-a.e. on $[U f < \infty]$.
\end{proof}

\begin{rem}  
The implications i) $\Leftrightarrow$ iii) $\Rightarrow$ iv) in Proposition \ref{prop2.3} 
should be compared to \cite{Fu07}, Proposition 1.1, ii) and \cite{FuOsTa11}, Lemma 1.6.4 and Theorem 4.7.1, ii), 
where the context is that of a strongly continuous sub-Markovian semigroup on $L^2(E, m)$, 
respectively of a symmetric Dirichlet form and the proofs are based on Hopf's maximal inequality.
\end{rem}

As a consequence we have the following useful result.  

\begin{coro} \label{coro 2.4} 
Assume that $\alpha U_\alpha^\ast 1 = 1$ $m$-a.e., $\alpha > 0$ and there exists an $m$-a.e. strictly positive $m$-integrable $\mathcal{U}$-excessive function. 
Then $\mathcal{U}$ is $m$-recurrent.

Consequently, if $m(E) < \infty$ then the following assertions are equivalent.

i) $\mathcal{U}$ is $m$-recurrent.

\vspace{0.1cm}

ii) $\alpha U_\alpha 1 = 1$ $m$-a.e., $\alpha > 0$.

\vspace{0.1cm}

iii) $\alpha U_\alpha^\ast 1 = 1$ $m$-a.e., $\alpha > 0$.
\end{coro}
\begin{proof}
Let $s \in \mathcal{E}(\mathcal{U}) \cap L^1(E,m)$, $s >0$ $m$-a.e. 
If $u \in \mathcal{E}(\mathcal{U})$ and $u_n := \inf(ns, u)$ for all $n \in \mathbb{N}^\ast$, 
then $u_n \nearrow u$ $m$-a.e. and $(u_n)_n \subset \mathcal{E}(\mathcal{U}) \cap L^1(E,m)$. 
Since
\[
\int{\alpha U_\alpha u_n dm} = \int{u_n \alpha U_\alpha^\ast 1 dm} = \int{u_n dm}, 
\] 
it follows that $m$-a.e. $\alpha U_\alpha u_n = u_n$ for all $n$  and therefore $\alpha U_\alpha u = u$. 
By Proposition \ref{prop2.3} we conclude that $\mathcal{U}$ is $m$-recurrent.

Assume now that $m(E) < \infty$. The implication i) $\Rightarrow$ ii) follows by Proposition \ref{prop2.3}.

ii) $\Rightarrow$ iii) Since $\int{\alpha U_\alpha^\ast 1 dm} = \int{\alpha U_\alpha 1 dm} = \int{1 dm} < \infty$ then condition iii) is satisfied. 
The implication iii) $\Rightarrow$ i) follows by the first part of the statement. 
\end{proof}

\begin{rem}  
a) By Theorem \ref{prop 2.9} below, the resolvent may be recurrent without possessing any excessive functions except the constant ones. 
In such situations the first assertion in Corollary \ref{coro 2.4} is not applicable unless $m(E)<\infty$. 

b) If $\mathcal{U}$ is the resolvent of a symmetric Dirichlet form $(\mathcal{E}, D(\mathcal{E}))$ on $L^2(E, m)$ 
then by \cite{FuOsTa11}, Theorem 1.6.3 (see also Theorem 1.6.5), $\mathcal{U}$ is $m$-recurrent 
if and only if there exists a sequence $(u_n)_{n} \subset D(\mathcal{E})$ such that $u_n \nearrow 1$ $m$-a.e. and $\lim\limits_{n} \mathcal{E}(u_n,u_n) = 0$. 
Hence, if $m(E) < \infty$, then $\mathcal{U}$ is $m$-recurrent if and only if $1 \in D(\mathcal{E})$ and $\mathcal{E}(1,1) = 0$, 
which is in fact a particular case of Corollary \ref{coro 2.4} (see also Corollary \ref{coro 5.3}).
\end{rem}

\begin{defi*}
A set $A \in \mathcal{B}$ is called {\it $\mathcal{U}${\it -absorbing}} (with respect to $m$) provided that
\[
U(1_{E \setminus A}) = 0 \quad m\mbox{-a.e. on} \; A.
\]
\end{defi*} 

\begin{rem} \label{rem 2.5} 
a) If the set $A$ is $\mathcal{U}$-absorbing (with respect to $m$) and $B \in \mathcal{B}$ is such that
$A = B$  $m${-a.e.} (i.e. $m(A \Delta B) = 0$),  then $B$ is also $\mathcal{U}$-absorbing.

b) If $\beta > 0$ then a set $A \in \mathcal{B}$ is simultaneously $\mathcal{U}$-absorbing and $\mathcal{U}_{\beta}$-absorbing.
\end{rem} 

\begin{prop} \label{prop2.4} 
The following assertions are equivalent for a set $A \in \mathcal{B}$.

i) The set $A$ is $\mathcal{U}$-absorbing (with respect to $m$).

ii) The set $E \setminus A$ is $\mathcal{U}^{\ast}$-absorbing (with respect to $m$).

iii) There exists a set $B \in \mathcal{B}$ such that $1_{E \setminus B} \in \mathcal{E}(\mathcal{U})$ and
\[
A = B \quad m \mbox{-a.e.}
\]

iv) There exists a $\mathcal{U}$-excessive function $u$ such that
\[
A = [u = 0] \quad m\mbox{-a.e.}
\]

v) There exists a $\mathcal{U}$-excessive function $u$ such that
\[
A = [u < +\infty] \quad m\mbox{-a.e.}
\]
\end{prop}
\begin{proof}
i) $\Leftrightarrow$ ii). If $U(1_{E \setminus A}) = 0$ on $A$ $m$-a.e. then $0 = \int 1_A U(1_{E \setminus A}) dm = \int 1_{E \setminus A} U^{\ast}(1_A) dm$, hence $U^{\ast}(1_A) = 0$ on $E \setminus A$ $m$-a.e.
Therefore the set $A$ is also $U^{\ast}$-absorbing.

i) $\Rightarrow$ iii) Let $B = [U(1_{E \setminus A}) = 0]$. By i) we have
\[
A \subset B \quad m\mbox{-a.e.}
\]
If we put $M := B\setminus A$ then $U(1_M) \leq U(1_{E \setminus A}) = 0$ on $M$ and by the complete maximum principle we deduce that $U(1_M) = 0$, $m(M) = 0$.
It follows that $B = A$ $m$-a.e. and since $U(1_{E\setminus A}) \in \mathcal{E}(\mathcal{U})$ we get also that $1_{E\setminus B} \in \mathcal{E}(\mathcal{U})$.

The implication iii) $\Rightarrow$ iv) is clear and iv) $\Rightarrow$ i) follows by assertion a) of Remark \ref{rem 2.5} since the set $[u = 0]$ is $\mathcal{U}$-absorbing.

iii) $\Rightarrow$ v). Let $B \in \mathcal{B}$ be such that $A = B$ $m$-a.e. and $1_{E \setminus B} \in \mathcal{E}(\mathcal{U})$. 
Then the function $u$ defined by
\[
u := \left\{\begin{array}{ll}
\infty & {\rm on} \; E \setminus B\\
0 & {\rm on} \; B
\end{array}
\right.
\]
is $\mathcal{U}$-excessive and clearly $B = [u < \infty]$.

v) $\Rightarrow$ i). Let $u \in \mathcal{E}(\mathcal{U})$ be such that $A = [u < +\infty]$ $m$-a.e. and put $B := [u < \infty]$. Then $U(1_{E \setminus B}) \leq \dfrac{1}{n}u$ on $E$ for all $n$, $U(1_{E \setminus B}) = 0$ on $B$.
Therefore $B$ is $\mathcal{U}$-absorbing.
\end{proof}

\begin{coro} \label{coro 2.7} 
If $(A_n)_{n}$ is a sequence of $\mathcal{U}$-absorbing sets then $\mathop{\bigcup}\limits_{n} A_n$ and $\mathop{\bigcap}\limits_{n} A_n$ are also $\mathcal{U}$-absorbing.
\end{coro}

\begin{proof}
By Proposition \ref{prop2.4} for every $n$ there exists $u_n \in \mathcal{E}(\mathcal{U})$ such that $A_n = [u_n = 0]$ $m$-a.e. 
Let $u := \mathop{\inf}\limits_{n} u_n$. 
Then $[u = 0] = \mathop{\bigcup}\limits_{n}[u_n = 0]$ and $\alpha U_{\alpha} u \leq u$ $m$-a.e. for all $\alpha > 0$.
From Proposition \ref{prop 2.2} and using again Proposition \ref{prop2.4} we conclude that $\mathop{\bigcup}\limits_{n}A_n$ is $\mathcal{U}$-absorbing.
The equivalence i) $\Leftrightarrow$ ii) in the above proposition implies now that $\mathop{\bigcap}\limits_{n} A_n$ is also $\mathcal{U}$-absorbing.
\end{proof}

\begin{defi*}
The resolvent $\mathcal{U} = (U_{\alpha})_{\alpha > 0}$ is named {\it $m$-irreducible} provided that there exists 
no nontrivial $\mathcal{U}$-absorbing set (with respect to $m$), i.e., if $A \in \mathcal{B}$ is $\mathcal{U}$-absorbing then either $m(A) = 0$ or $m(E \setminus A) = 0$.
\end{defi*}

By Proposition \ref{prop2.4} it follows that $\mathcal{U}$ and $\mathcal{U}^{\ast}$ are simultaneously $m$-irreducible.

The next result expresses the dichotomy of $\mathcal{U}$ under the assumption of irreducibility.

\begin{prop} \label{prop 2.5} 
Assume that $\mathcal{U}$ is $m$-irreducible.

Then the resolvent $\mathcal{U} = (U_{\alpha})_{\alpha > 0}$ is either $m$-transient or $m$-recurrent.
\end{prop}
\begin{proof}
Suppose that $\mathcal{U}$ is not $m$-recurrent, then there exists $f \in p\mathcal{B} \cap L^1(E, m)$ such that $m([0 < U f < \infty]) > 0$. 
Then $m([U f< \infty]) > 0$ and $m([U f > 0]) > 0$ and by Proposition \ref{prop2.4} the sets $[U f < +\infty]$ and $[U f = 0]$ are $m$-absorbing. 
Since $\mathcal{U}$ is $m$-irreducible, we deduce that $0 = m([U f = +\infty]) = m([U f = 0])$. 
Therefore we have $m$-a.e. $U f < \infty$ and $f > 0$, hence $\mathcal{U}$ is $m$-transient.
\end{proof}

\begin{rem}  
Proposition \ref{prop 2.5} was already proved in the case of a strongly continuous sub-Markovian resolvent on $L^2(E,m)$ and we refere to \cite{Fu07}, Theorem 1.1, i) and \cite{FuOsTa11}, Lemma 1.6.4 iii). 
We also recall that in the case of a convolution semigroup on $\mathbb{R}^d$ then by \cite{Fu07}, Theorem 1.2, the dichotomy still holds without requiring irreducibility.
\end{rem}

Recall that if $\beta > 0$ and $f \in p\mathcal{B}$ we may consider the ($\beta$-order) {\it reduced function} of $f$, defined by
\[
R_{\beta} f := \inf \{ v \in \mathcal{S}(\mathcal{U}_{\beta}) : v \geq f\}.
\]

Due to a result of Mokobodzki (see for example \cite{BeBo04}, Theorem 1.1.9) we have that $R_{\beta} f$ is $\mathcal{B}$-measurable and it is $\mathcal{U}_{\beta}$-supermedian. 
Notice that if $\beta' < \beta$ then $\mathcal{S}(\mathcal{U}_{\beta'})\subset \mathcal{S}(\mathcal{U}_{\beta})$ and consequently $R_{\beta} f \leq R_{\beta'} f$. 
Therefore if $f \in p\mathcal{B}$, we may consider $R_0 f$, the $0$-order reduced function of $f$, defined by
\[
R_0 f(x) := \mathop{\sup}\limits_{\beta} R_{\beta}f(x) = \mathop{\lim}\limits_{\beta \searrow 0} R_{\beta} f(x), \; x \in E.
\]
It follows that $R_0 f$ is an $\mathcal{U}$-supermedian function. 
If $u \in \mathcal{S}(\mathcal{U})$ and $A \in \mathcal{B}$ then $R_0 (1_{A}u)\in \mathcal{S}(\mathcal{U}) = \mathop{\bigcap}\limits_{\beta > 0} \mathcal{S}(\mathcal{U}_{\beta})$, it is dominated by $u$ and equal to $u$ on $A$. 
Therefore
\[
R_0 (1_Au) = \inf \{ v\in \mathcal{S}(\mathcal{U}) : v\geq u \; {\rm on} \; A \}.
\]

As we mentioned above, the resolvent may not possess $0$-order excessive functions other than the constant functions (with respect to $m$). 
This is the case if and only if the resolvent is irreducible recurrent and we express this fact in the next proposition (for equivalence i) $\Leftrightarrow$ ii) below see also \cite{Fu07}, Theorem 1.1, ii)). 
For a probabilistic approach (in terms of an $m$-symmetric right process) of implication i) $\Rightarrow$ iv) we refer to \cite{FuOsTa11}, Theorem 4.7.1, where condition iv) below holds q.e. (i.e. outside some $m$-polar set) and not only $m$-a.e.

\begin{prop} \label{prop 2.9} 
The following assertions are equivalent.

i) $\mathcal{U}$ is $m$-irreducible and $m$-recurrent.

\vspace{0.1cm}

ii) For every $f \in p\mathcal{B} \cap L^1(E, m)$ with $\int f dm > 0$ we have $U f = +\infty$ $m$-a.e.

\vspace{0.1cm}

iii) If $ f \in p\mathcal{B}$ then $m$-a.e. we have either $U f = 0$ or $U f = +\infty$.

\vspace{0.1cm}

iv) We have $m$-a.e. that every $\mathcal{U}$-excessive function is constant and $\alpha U_{\alpha} 1 = 1$, $\alpha > 0$.
\end{prop}
\begin{proof}
i) $\Rightarrow$ ii). Let $f \in p\mathcal{B} \cap L^1(E, m)$ with $\int f dm > 0$ and $A := [f > 0]$, then $m(A) > 0$ and by Proposition \ref{prop2.3} (since $\mathcal{U}$ is $m$-recurrent) it follows that
\[
m([U f = +\infty]) \geq m([U f = +\infty] \cap A) > 0.
\]
The set $[U f < +\infty]$ is $\mathcal{U}$-absorbing (c.f. Proposition \ref{prop2.4}) and therefore, $\mathcal{U}$ being $m$-irreducible, we deduce that $m([U f < +\infty]) = 0$, hence $U f = +\infty$ $m$-a.e.

ii) $\Rightarrow$ iii). Let $g \in pb\mathcal{B} \cap L^1(E, m)$, $g > 0$. 
If $f \in p\mathcal{B}$ and $f_n := \inf(f, ng)$, $n \in \mathbb{N}^{\ast}$, then $(f_n)_n \subset p\mathcal{B} \cap L^1(E, m)$ and $f_n \nearrow f$ pointwise. 
If $f \in p\mathcal{B}$ and $m([U f > 0]) > 0$ then $m(U f) > 0$ and therefore $\int f dm > 0$. 
We consider $n_0 \in \mathbb{N}^{\ast}$ such that $\int f_{n_0} dm > 0$ and by hypothesis ii) we get $m$-a.e. $U f \geq U f_{n_0} = +\infty$.

iii) $\Rightarrow$ i). Let $f \in b\mathcal{B} \cap L^1(E, m)$, $f > 0$. Then $U f > 0$ and therefore by iii) we deduce that $U f = +\infty$ $m$-a.e.
From Proposition \ref{prop2.3} we conclude that $\mathcal{U}$ is $m$-recurrent.

Let now $A$ be $m$-absorbing. We may assume that $1_{E \setminus A} \in \mathcal{E}(\mathcal{U})$ (see Proposition \ref{prop2.4}). 
Therefore we have $U(1_{E \setminus A}) =0$ on $A$ and $U(1_{E \setminus A}) > 0$ on $E \setminus A$.
If $m(A) > 0$ then by hypothesis iii) we get $U(1_{E \setminus A}) = 0$ $m$-a.e. and therefore $m(E \setminus A) = 0$.

i) $\Rightarrow$ iv). Let $u \in \mathcal{E}(\mathcal{U})$ be such that $\int_{E} u dm > 0$. 
We may assume that $u \leq 1$ $m$-a.e. and notice that if $v \in \mathcal{S}(\mathcal{U})$ and $v \leq u$ $m$-a.e. then (cf. Proposition \ref{prop2.3}) there exists $w \in b\mathcal{S}(\mathcal{U})$ such that $u = v + w$ $m$-a.e.

Let $G \in \mathcal{B}$ such that $m(G) > 0$. We claim that 
\[
R_0(1_{G}u) = u \quad m\mbox{-a.e.},
\]
Indeed, if $w \in b\mathcal{S}(\mathcal{U})$ is such that $u = R_0(1_G u) + w$ $m$-a.e., because $R_0(1_G u) = u$ on $G$ we get that $[w = 0] \supset G$ $m$-a.e., hence $m([w = 0]) \geq m(G) > 0$. 
Since $\mathcal{U}$ is $m$-irreducible we conclude that $w = 0$ $m$-a.e.

For every $\alpha \in (0, 1]$ we consider the set $G_{\alpha} \in \mathcal{B}$ defined by
$G_{\alpha} := [u > \alpha].$
By the above considerations we deduce that if $m(G_{\alpha}) > 0$ then $m$-a.e. we have $R_0(1_{G_{\alpha}} u) = u$ and $R_0 1_{G_{\alpha}} = 1$. 
From $\alpha \leq u$ on $G_{\alpha}$ it follows that $\alpha \leq u$ $m$-a.e. on $E$.

Let further
$\alpha_0 := \sup\{ \alpha > 0 : m(G_{\alpha}) > 0 \}.$ 
Then $u \geq \alpha_0$ $m$-a.e. and $m(G_{\alpha}) = 0$, hence $u = \alpha_0$ $m$-a.e.

iv) $\Rightarrow$ i). By Proposition \ref{prop2.3} it follows clearly that $\mathcal{U}$ is $m$-recurrent. 
If $A$ is $m$-absorbing then by Proposition \ref{prop2.4} there exists $B \in \mathcal{B}$ such that $A = B$ $m$-a.e. and $1_{E \setminus B} \in \mathcal{E}(\mathcal{U})$.
Since by hypothesis the function $1_{E \setminus B}$ should be $m$-a.e. a constant, we get that either $m(B) = 0$ or $m(E \setminus B) = 0$. 
Therefore $\mathcal{U}$ is $m$-irreducible.
\end{proof}

\noindent
{\bf Transience, recurrence, and irreducibility of a right process.}
In this subsection $\mathcal{U}$ is the resolvent of a right (Markov) process $X=(\Omega, \mathcal{F}, \mathcal{F}_t, X_t, P^x)$ with values in $E$, and $m$ is a sub-invariant $\sigma$-finite measure.

\begin{defi*} (cf. \cite{Ge80}) 
i) The resolvent $\mathcal{U}$ (or the process $X$) is called {\it transient} provided there exists a strictly positive Borel measurable function $f$ such that $Uf <\infty$.

ii) The resolvent $\mathcal{U}$ (or the process $X$) is called {\it recurrent} if $U1_B=0$ or $U1_B=\infty$ for all $B \in \mathcal{B}$.
\end{defi*}

\begin{rem}   
i) By \cite{Ge80}, Proposition 2.2 and Proposition 2.4, the following probabilistic characterizations hold: i.1) $\mathcal{U}$ is transient if and only if there exists a sequence of Borel finely open sets $(B_n)_{n \geq 1}$ increasing to $E$ such that the last exit time of $B_n$ is finite $P_x$-a.e. for all $x \in E$.
i.2) $\mathcal{U}$ is recurrent if and only if any excessive function is constant, and furthermore, if and only if the last exit time of any finely open set is infinite almost surely. 

ii) Following the lines of \cite{FuOsTa11}, Lemma 4.8.1 one can show that recurrence as defined above is, as a matter of fact, equivalent with the (apparently stronger) so called Harris recurrence: 
$\int\limits^\infty_0 1_B(X_s)ds = \infty$ $P^x$-a.s. for all $x \in E$ whenever $B \in \mathcal{B}$ with $U(1_B) > 0$.
\end{rem}

Recall that a set $A \in \mathcal{B}$ is called {\it absorbing} if there exists an excessive function $v \in p\mathcal{B}$ such that $A=[v=0]$. 
We remark that $A$ is absorbing if and only if $1_{E \setminus A}$ is excessive, and if and only if there exists an excessive function $v \in p\mathcal{B}$ such that $A=[v < \infty]$. If $B \in \mathcal{B}$ is $m$-negligible  such that $E \setminus B$ is absorbing then the set $B$ is named $m${\it -inessential}.

As in \cite{BeRo11}, Section 3, if $A \in \mathcal{B}$ such that $E \setminus A$ is $m$-inessential, then we may consider the following two modifications of $\mathcal{U}$:

- the {\it restriction} $\mathcal{U}{'}$ of $\mathcal{U}$ on $A$, i.e. the sub-Markovian resolvent of kernels on $(A, \mathcal{B}|_A)$ defined as:
\[
U'_\alpha  f = U_\alpha \overline{f} |_A  \; \; \; \mbox{for all } f  \in p\mathcal{B}|_A,
\]
where $\overline{f} \in p\mathcal{B}$ is such that $\overline{f}|_A = f$. 

- the {\it ($1$-order) trivial modification of $\mathcal{U}$ on A}, namely the sub-Markovian resolvent
$\mathcal{U}^A = (U^A_\alpha)_{\alpha > 0}$ on $(E, \mathcal{B})$ defined by
\[
U^A_\alpha f = 1_A U_\alpha (f 1_A) + \frac{1}{1+\alpha} f 1_{E \setminus A} \; \; \; \alpha > 0, f \in p\mathcal{B}.
\]
Then both of the above resolvents induced by $\mathcal{U}$ and $A$ are the resolvents of some right processes with state spaces $(A, \mathcal{B}|_A)$, respectively $(E, \mathcal{B})$. 

\begin{rem} \label{rem 2.1.15}  
i) $\mathcal{U}^A$ is an $m$-version of $\mathcal{U}$, that is $U^A_\alpha f = U_\alpha f$, $\alpha > 0$, $m$-a.e. for all $f \in p\mathcal{B}$.

ii) $\mathcal{U}$ is $m$-transient, $m$-recurrent, or $m$-irreducible if and only if $\mathcal{U}{'}$, and hence $\mathcal{U}^A$, are $m$-transient, $m$-recurrent, or $m$-irreducible, respectively.
\end{rem}

\begin{prop} \label{prop 2.1.14}  
The following assertions are equivalent.

i) $\mathcal{U}$ is $m$-transient.

\vspace{0.1cm}

ii) There exists a Borel set $A$ such that $E \setminus A$ is $m$-inessential and the $1$-order trivial modification $\mathcal{U}^A$ is transient.
\end{prop}
\begin{proof}
Since the implication ii) $\Rightarrow$ i) is clear, we prove only the converse. 
Let $f_0 \in p\mathcal{B}$ such that $m$-a.e. we have that $f_0 > 0$ and $Uf_0 < \infty$.
Clearly, we may assume that $f_0 > 0$ on E.
If $A:= [Uf_0 < \infty]$ then $m(E \setminus A) = 0$ hence $E \setminus A$ is $m$-inessential.
Finally, if $p\mathcal{B} \ni f_1 = f_0$ on $A$ then $U^A f_1 = 1_A U(f_0 1_A) + \frac{1}{1+\alpha}f_{1}1_{E \setminus A} < \infty$ on $E$.
\end{proof}

We say that $\mathcal{U}$ is {\it irreducible} if  for any absorbing set $A \in \mathcal{B}$ we have either $A = \emptyset$ or $A = E$.

\begin{prop} \label{prop 2.1.16}  
The following assertions are equivalent.

i) $\mathcal{U}$ is irreducible.

\vspace{0.1cm}

ii) For every $\alpha$-sub-invariant measure $\mu$ we have that $\mu$ is a reference measure, $\alpha \geq 0$.

\vspace{0.1cm}

iii) $\mathcal{U}$ is $m$-irreducible and $m$ is a reference measure.
\end{prop}
\begin{proof}
i) $\Rightarrow$ ii). Let $\mu$ be an $\alpha$-sub-invariant measure and $A \in \mathcal{B}$ an $\mu$-negligible set. 
Then $[U_\alpha 1_A = 0] = E$ $\mu$-a.e., and because $\mathcal{U}$ is irreducible we get that $U_\alpha 1_A = 0$.

ii) $\Rightarrow$ iii). Clearly, we only have to check that $\mathcal{U}$ is $m$-irreducible. 
If $A \in \mathcal{B}$ is $m$-absorbing such  that $m(A) > 0$ then there exists $x \in A$ such that $U_\alpha 1_{E \setminus A}(x) = 0$. 
Since the measure $\delta_x \circ U_\alpha$ is a reference measure (as an $\alpha$-sub-invariant measure) it follows that $U_\alpha 1_{E \setminus A} = 0$. 
But by Proposition \ref{prop 2.2} there exists an excessive $m$-version $v$ of $1_{E \setminus A}$. 
Consequently, $v= \sup\limits_{\alpha} \alpha U_\alpha v = \sup\limits_{\alpha} \alpha U_\alpha 1_{E \setminus A} = 0$, hence $m(E \setminus A) = 0$

iii) $\Rightarrow$ i). Let $v \in p\mathcal{B}$ be an excessive function and $A:=[v=0]$. 
In particular, we have that $A$ is $m$-absorbing. If $m(A)=0$ then $U_\alpha 1_A =0$, hence
$1_{E\setminus A} \geq \alpha U_\alpha 1_{E\setminus A}=\alpha U_\alpha 1 \nearrow 1$.
It follows that $A = \emptyset$.
Now assume that $m(E \setminus A)=0$ so that $U_\alpha 1_{E \setminus A} =0$. 
Because $E \setminus A$ is finely open, by \cite{BeBo04}, Proposition 1.3.2 we have that
$\liminf\limits_{\alpha \to \infty} \alpha U_\alpha 1_{E \setminus A} = 1$ on $E \setminus A$.
In conclusion, $E \setminus A= \emptyset$.
\end{proof}

Let $\mu$ be a $\sigma$-finite measure on $(E, \mathcal{B})$. 
As in \cite{BeBo97}, we say that the $\mu${\it -quasi-Lindel\"of property
holds} (for the {\it fine topology} on $E$, which is the coarsest topology on $E$ making continuous all $\alpha$-excessive functions) if: 
for any collection $\mathcal{G}$ of finely open Borel
subsets of $E$ there exists a countable subcollection $(G_k)_{k \in \mathbb{N}}$ such that the set
$\bigcup\limits_{G \in \mathcal{G}}G \setminus \bigcup\limits_{k \in \mathbb{N}} G_k$
is $\mu$-semipolar. 
If $\bigcup\limits_{k \in \mathbb{N}} G_k$ differs from $\bigcup\limits_{G \in \mathcal{G}}G$ by 
a semipolar set, then we say that the {\it quasi-Lindel\"of property} holds.

\begin{rem}  
It is known that the quasi-Lindel\"of property holds if and only if $\mathcal{U}$ posses a {\it reference} measure
(i.e. there exists a $\sigma$-finite measure $\lambda$ such that $Uf=0$ whenever $\lambda(f)=0$ for $f \in p\mathcal{B}$).
Also, the $m$-quasi-Lindel\"of property holds if and only if there exists a set $A \in \mathcal{B}$ such that $E \setminus A$
is $m$-inessential and the restriction of $m$ to $A$ is a reference measure for the restriction 
of $\mathcal{U}$
on $A$ (see \cite{BeBo97}, Section 3, and the references therein). 
We reiterate that $m$ is a sub-invariant measure, fixed at the beginning of this subsection.
\end{rem}

\begin{prop} \label{prop 2.1.17}  
The following assertions are equivalent.

i) The $m$-quasi-Lindel\"of property holds for $\mathcal{U}$ and $\mathcal{U}$ is $m$-irreducible.

\vspace{0.1cm}

ii) There exists a Borel set $A$ such that $E \setminus A$ is $m$-inessential and the restriction of $\mathcal{U}$ to $A$ is irreducible.
\end{prop}

\begin{proof}
i) $\Rightarrow$ ii). If the $m$-quasi-Lindel\"of property holds for $\mathcal{U}$ then, by \cite{BeBo97}, Theorem 3.1, there exists a Borel set $A$ such that $E \setminus A$ is $m$-inessential and $m$ is a reference measure for $\mathcal{U}{'}$. 
But $\mathcal{U}{'}$ is $m$-irreducible so assertion ii) follows by Proposition \ref{prop 2.1.16}.

The implication ii) $\Rightarrow$ i) follows by Proposition \ref{prop 2.1.16} and \cite{BeBo97}, Theorem 3.1.
\end{proof}

The next result is a version of Lemma 2.1 from \cite{BeBoRo06a}.

\begin{lem} \label{lem 2.1.15}  
If $E_0 \in \mathcal{B}$ is finely closed and $m(E \setminus E_0) = 0$ then there exists a set $F \subset E_0$ such  that $E \setminus F$ is $m$-inessential.  
\end{lem}

\begin{proof}
Let $(E_n)_{n \geq 1} \subset \mathcal{B}$ be the sequence defined inductively by $E_{n+1} = E_n \cap [U(1_{E \setminus E_n}) = 0]$ if
$n \geq 0$. If $F := \bigcap\limits_n E_n$ then $\mathcal{B} \ni F \subset E_0$, $m(E \setminus F) = 0$, and $U1_{E \setminus F}=0$ on $F$.
Moreover, $F$ is finely closed, as an intersection of finely closed sets. 
Therefore, the function $1_{E \setminus F}$ is supermedian and finely lower semicontinuous. 
By \cite{BeBo04}, Corollary 1.3.4 we get that $1_{E \setminus F}$ is excessive. 
Clearly, $E \setminus F$ is $m$-inessential.
\end{proof}

\begin{prop} \label{prop 2.1.18} 
The following assertions are equivalent.

i) The $m$-quasi-Lindel\"of property holds for $\mathcal{U}$ and $\mathcal{U}$ is $m$-recurrent and $m$-irreducible.

\vspace{0.1cm}

ii) There exists a Borel set $A$ such that $E \setminus A$ is $m$-inessential and the restriction of $\mathcal{U}$ to $A$ is recurrent.
\end{prop}
\begin{proof}
i) $\Rightarrow$ ii). By Proposition \ref{prop 2.1.17} and Remark \ref{rem 2.1.15} there exists a Borel set $A$ such that $E \setminus A$ is $m$-inessential and the restriction $\mathcal{U}{'}$ is irreducible and $m$-recurrent.
Therefore, if $B \in \mathcal{B}|_A$ then $U{'} 1_B = 0$ or $U{'} 1_B = \infty$, $m$-a.e.
Let $E_0:=[U{'} 1_B = 0]$ such that $m(A \setminus E_0)=0$.
From Lemma \ref {lem 2.1.15} there exists a non-empty absorbent set $F \subset E_0$. Consequently, $E_0 = A$. 
The other case is similar.   

ii) $\Rightarrow$ i). Since $m(E \setminus A) = 0$, by Remark \ref{rem 2.1.15} it follows that all $\mathcal{U}$-excessive functions are constant $m$-a.e., and by Proposition \ref{prop 2.9} we obtain that $\mathcal{U}$ is $m$-recurrent and $m$-irreducible. 
The $m$-quasi-Lindel\"of property follows by Proposition \ref{prop 2.1.17}.
\end{proof}

\vspace{0.2cm}

\noindent
{\bf Irreducibility and invariance.} 
As in \cite{AlKoRo97a} a real-valued function $v \in \bigcup\limits_{1 \leq p \leq \infty} L^p(E, m)$ is called {\it $\mathcal{U}$-invariant} (with respect to $m$) provided that for all $\alpha > 0$ and $f \in bp\mathcal{B}$ we have 
\[
U_{\alpha}(vf) = v U_{\alpha} f \quad m\mbox{-a.e.} 
\]

A set $A \in \mathcal{B}$ is called {\it $\mathcal{U}$-invariant} if the function $1_A$ is $\mathcal{U}$-invariant. 
It is easy to check that the collection of all $\mathcal{U}$-invariant sets is a $\sigma$-algebra.

\begin{rem} \label{rem 2.14} 
Let $v$ be a $\mathcal{U}$-invariant function.

i) If $u$ is a $\mathcal{B}$-measurable real-valued function and $u = v$ $m$-a.e. then $u$ is also $\mathcal{U}$-invariant.

ii) If $v \geq 0$ then there exists a $\mathcal{U}$-excessive function $u$ such that $u = v$ $m$-a.e. 
If in addition $\alpha U_{\alpha} 1 = 1$ $m$-a.e. then $\alpha U_{\alpha} v = v$ $m$-a.e. 
Indeed, the assertion follows since $\alpha U_{\alpha} v = v\alpha U_{\alpha} 1 \leq v$ $m$-a.e.
\end{rem}

For every $p \in [1, \infty]$ let $\mathcal{A}_p$ be the set of all $\mathcal{U}$-invariant functions from $L^p(E, m)$.

\begin{prop} \label{prop 2.15} 
The set $\mathcal{A}_p$, $1 \leq p \leq \infty$ is a vector lattice with respect to the pointwise order relation.
\end{prop}
\begin{proof}
It is clear that $\mathcal{A}_p$ is a vector space. If $u \in \mathcal{A}_p$, $\alpha > 0$ and $f \in bp\mathcal{B}$ then we have $m$-a.e.
\[
U_{\alpha}(u^{+}f) = U_{\alpha}(1_{[u>0]}u f) = uU_{\alpha}(1_{[u > 0]}f) \leq u^{+}U_{\alpha} f.
\]
Consequently we have also $U_{\alpha}(u^{-} f) \leq u^{-}U_{\alpha} f$ and therefore
\[
U_{\alpha}(|u|f) \leq |u|U_{\alpha} f \quad m\mbox{-a.e.}
\]
On the other hand we have $m$-a.e.
\[
\pm u U_{\alpha} f = U_{\alpha} (\pm u f) \leq U_{\alpha}(|u|f),
\]
and thus $|u|U_{\alpha}f \leq U_{\alpha}(|u|f)$, hence $|u|\in \mathcal{A}_p$.
\end{proof}

\begin{prop} \label{prop 2.16} 
The following assertions are equivalent for a real-valued function $u \in L^p(E, m)$.

i) $u$ is $\mathcal{U}$-invariant.

\vspace{0.1cm}

ii) $u$ is $\mathcal{U}^{\ast}$-invariant.

\vspace{0.1cm}

iii) For all $f, g\in bp\mathcal{B} \cap L^{p'}(E, m)$ and $\alpha > 0$ we have
\[
\int f u U_{\alpha}^{\ast} g dm = \int g u U_{\alpha} f dm.
\]
\end{prop}
\begin{proof}
Notice that $u \in \mathcal{A}_p$ if and only if for all $f,g \in bp\mathcal{B} \cap L^{p'}(E, m)$ we have
\[
\int g U_{\alpha}(uf) dm = \int gu U_{\alpha} f dm.
\] 
The equivalence i) $\Leftrightarrow$ iii) follows now since 
\[
\int g U_{\alpha}(uf) dm = \int f u U_{\alpha}^{\ast} g dm.
\]

We have also i) $\Leftrightarrow$ ii) since property iii) is the same for $\mathcal{U}$ and $\mathcal{U}^{\ast}$.
\end{proof}

\begin{coro} \label{coro 2.17} 
If $A \in \mathcal{B}$ then the following assertions are equivalent.

i) The function $1_A$ is $\mathcal{U}$-invariant.

\vspace{0.1cm}

ii) The sets $A$ and $E \setminus A$ are both of them $\mathcal{U}$-absorbing.

\vspace{0.1cm}

iii) There exists a function $s \in \mathcal{E}(\mathcal{U}) \cap \mathcal{E}(\mathcal{U^\ast}$) such that $A = [s=0]$  $m$-a.e.

\vspace{0.1cm}

iv) There exists an $\mathcal{U}$-invariant function $s$ such that $A = [s=0]$  $m$-a.e.
\end{coro}

The next main theorem collects several characterizations of invariance and also shows that, like absorbance, invariance is determined by only one operator $U_\alpha$. 

Let
\[
\mathcal{I}_p : = \{ u \in L^p(E, m) : \alpha U_\alpha u = u \; m\mbox{-a.e.}, \; \alpha > 0 \}.  
\]

\begin{thm} \label{thm 2.19} 
Let $u \in L^p(E, m)$, $1 \leq p < \infty $ and consider the following conditions.

\vspace{0.1cm}

i) $\alpha U_{\alpha}u = u$ $m$-a.e. for one (and therefore for all) $\alpha > 0$.

\vspace{0.2cm}

ii) $\alpha U_{\alpha}^\ast u = u$ $m$-a.e., $\alpha > 0$. 

\vspace{0.1cm}

iii) The function $u$ is $\mathcal{U}$-invariant.

\vspace{0.1cm}

iv) $U_\alpha u = u U_\alpha 1$ and $ U_\alpha^\ast u = u U_\alpha^\ast 1$ $m$-a.e. for one (and therefore for all) $\alpha > 0$.

\vspace{0.1cm}

v) The function $u$ is measurable w.r.t. the $\sigma$-algebra of all $\mathcal{U}$-invariant sets.

\vspace{0.1cm}

Then $\mathcal{I}_p$ is a vector lattice w.r.t. the pointwise order relation and i) $\Leftrightarrow$ ii) $\Rightarrow$ iii) $\Leftrightarrow$ iv) $\Leftrightarrow$ v).

If $\alpha U_\alpha 1 = 1$ or $\alpha U^\ast_\alpha 1 = 1$ $m$-a.e. then assertions i) - v) are equivalent.

If $m(E) < \infty$ and $p = \infty$ then all of the statements above are still true.
\end{thm}
\begin{proof}
i) $\Leftrightarrow$ ii) $\Rightarrow$ iii) and $\mathcal{I}_p$ is a vector lattice. 
It is clear that $\mathcal{I}_p$ is a vector space. If $u \in \mathcal{I}_p$ and $c$ is a positive real number then $m$-a.e. $\alpha U_\alpha (u-c)^+ \geq \alpha U_\alpha (u-c) \geq u - c$, hence $\alpha U_\alpha (u-c)^+ \geq (u-c)^+$ and by H\"older inequality we get $m$-a.e.
\[
\alpha U_\alpha ((u-c)^+)^p \geq \alpha^p U_\alpha ((u-c)^+)^p (U_\alpha 1)^{p-1} \geq (\alpha U_\alpha (u-c)^+)^p \geq ((u-c)^+)^p.
\]
Since $(u-c)^+ \in L^p(E, m)$ we have
\[
\int{((u-c)^+)^p dm} \leq \int{\alpha U_\alpha ((u-c)^+)^p dm} = \int{((u-c)^+)^p \alpha U_\alpha^\ast 1 dm} \leq \int{((u-c)^+)^p dm},
\]
therefore $m$-a.e.
\begin{equation} \label{eq 2.2}
\alpha U_\alpha ((u-c)^+)^p = ((u-c)^+)^p, \alpha U_\alpha (u-c)^+ = (u-c)^+ \;{\rm and} \; \alpha U_\alpha (u-c)^- \leq (u-c)^-.
\end{equation}
If we take $c = 0$ in the second relation in (\ref{eq 2.2}) it yields that $\mathcal{I}_p$ is a vector lattice hence we may assume that $u$ is positive. 
Furthermore, by Proposition \ref{prop2.4} it follows that the sets $[u \leq c] = [(u-c)^+ = 0]$ and $[ u \geq c] = [(u-c)^- = 0]$ are $\mathcal{U}$-absorbing for any $c \in \mathbb{R}_+$. 
Because $[u > c] = \bigcap\limits_{n=1}^\infty [u \geq c + \frac{1}{n}]$ from Corollary \ref{coro 2.7} and Corollary \ref{coro 2.17} we obtain that $1_{[u \leq c]}$ and consequently $1_{[b < u \leq c]}$  are $\mathcal{U}$-invariant for avery $b, c \in \mathbb{R}_+^\ast$. 
By approximating $u$ with linear combinations of functions of type $1_{[b < u \leq c]}$ and using monotone convergence we deduce that $u$ is $\mathcal{U}$-invariant and the implication i) $\Rightarrow$ iii) is proved. We continue by showing that i) implies ii) (the converse follows by duality). 
For simplicity, let us generically  write $A : = [b < u \leq c] \subset [u > 0]$, $b, c \in \mathbb{R}_+^\ast$ and recall that $u \in \mathcal{I}_p$ and $1_A$ are $\mathcal{U}$-invariant. 
In particular, we have $m$-a.e. that $U_\alpha u = u U_\alpha 1$, $\alpha U_\alpha 1 = 1$ on $[u > 0] \supset A$, $\alpha U^\ast_\alpha 1_A \leq 1_A$ and the function $1_A$ is integrable. 
Then
\[
\int{1_A dm} \geq \int{\alpha U_\alpha^\ast 1_A dm} = \int{1_A \alpha U_\alpha 1 dm} = \int{1_A dm},
\]
hence $\alpha U_\alpha^\ast 1_A = 1_A$ $m$-a.e. and again by approximating with step functions we conclude that $\alpha U_\alpha^\ast u = u$ $m$-a.e.

\vspace{0.2cm}

Clearly iii) implies iv).

iv) $\Rightarrow $ v). Assume that $u$ satisfies iv) for one $\alpha > 0$. 
If $c$ is a positive real number then $(u-c)^+ \in L^p(E, m)$ and $U_\alpha (u-c)^+ \geq U_\alpha (u-c) = (u-c) U_\alpha 1$, hence $U_\alpha (u-c)^+ \geq (u-c)^+ U_\alpha 1$ $m$-a.e. Moreover, since $U_\alpha ((u-c)^+)^p (U_\alpha 1)^{p-1} \geq (U_\alpha (u-c)^+)^p \geq ((u-c)^+)^p (U_\alpha 1)^p$ we have that $U_\alpha ((u-c)^+)^p \geq ((u-c)^+)^p U_\alpha 1$ $m$-a.e. 
Analogously, we get $U_\alpha^\ast (u-c)^+ \geq (u-c)^+ U_\alpha^\ast 1$ and $U_\alpha^\ast ((u-c)^+)^p \geq ((u-c)^+)^p U_\alpha^\ast 1$ $m$-a.e. 
Then
\[
\int{U_\alpha ((u-c)^+)^p +U_\alpha^\ast ((u-c)^+)^p dm} = \int{((u-c)^+)^p (U_\alpha^\ast 1 + U_\alpha 1) dm} \leq
\]
\[
\leq \int{U_\alpha ((u-c)^+)^p +U_\alpha^\ast ((u-c)^+)^p dm}
\]
which implies $m$-a.e.
\begin{equation} \label{eq 2.3}
U_\alpha (u-c)^+ = (u-c)^+ U_\alpha 1 \; {\rm and} \; U_\alpha^\ast (u-c)^+ = (u-c)^+ U_\alpha^\ast 1.
\end{equation}
Then $U_\alpha \inf (n(u-c)^+, 1) \leq \inf ( n(u-c)^+, 1) U_\alpha 1$ $m$-a.e. and letting $n$ tend to infinity we get $ U_\alpha 1_{[u > c]} \leq 1_{[u > c]} U_\alpha 1$ and analogously, $ U^\ast_\alpha 1_{[u > c]} \leq 1_{[u > c]} U^\ast_\alpha 1$ $m$-a.e. 
By Remark \ref{rem 2.5}, b) it follows that $[u \leq c]$ is $\mathcal{U}$-invariant. 
Taking $c = 0$ in (\ref{eq 2.3}) we obtain that the set of functions satisfying condition iv) is a vector lattice w.r.t. the pointwise order relation so we may assume that $u$ is positive. 
It follows that condition v) holds.

v) $\Rightarrow$ iii). Since $\mathcal{A}_p$ is a lattice we may assume that $u$ is positive. 
If $u$ satisfies v) then it can be approximated by an increasing sequence of invariant simple functions and by monotone convergence we conclude that $u$ is $\mathcal{U}$-invariant.

Finally, if $\alpha U_\alpha 1 = 1$ or $\alpha U^\ast_\alpha 1 = 1$ $m$-a.e. then $u \in \mathcal{A}_p$ if and only if $u \in \mathcal{I}_p$ and all of the assertions are equivalent. 
\end{proof} 

\begin{rem} \label{rem 2.16} 
i)  Similar characterizations for invariance as in Theorem \ref{thm 2.19}, but in the recurrent case and for functions which are bounded or integrable with bounded negative parts, as well as the fact that, in terms of semigroups (assuming a strong analyticity assumption), absorbance and invariance are determined by only one operator, were already obtained in \cite{Sc04}.

ii)  If $u \in L^p(E, m)$, $1 \leq p < \infty$ is in $\mathcal{I}_p$ (resp. is $\mathcal{U}$-invariant) then $\inf{(u, c)}$ is in $\mathcal{I}_p$ (resp. is $\mathcal{U}$-invariant) for all positive real numbers $c$. 
This is true by relation $(\ast)$ (resp. $(\ast \ast)$) ( see the proof of Theorem \ref{thm 2.19}) and the fact that $\inf{(u, c)} = u - (u - c)^+$.

iii) A set $A \in \mathcal{B}$ is $\mathcal{U}$-invariant if and only of  $U_\alpha 1_A = 1_A U_\alpha 1$ since the last equality implies that $U_\alpha 1_{E \setminus A} = 1_{E \setminus A} U_\alpha 1$ $m$-a.e. hence $A$ and $E \setminus A$ are $\mathcal{U}$-absorbing. 
However, if $u \in L^p(E, m)$, $1 \leq p \leq \infty$ we do not know if $U_\alpha u = u U_\alpha 1$ $m$-a.e. (without assuming $ U_\alpha^\ast u = u U_\alpha^\ast 1$ $m$-a.e. as in condition v) of the above theorem ) is enough for $u$ to be $\mathcal{U}$-invariant.
\end{rem}

If $p=\infty=m(E)$ then we have the following version of Theorem \ref{thm 2.19}.

\begin{prop} \label{prop 2.17} 
If $\mathcal{U}$ is $m$-recurrent then the following assertions are equivalent for a function $u \in L^\infty (E, m)$.

i) $\alpha U_\alpha u = u$ $m$-a.e., $\alpha > 0$.

\vspace{0.1cm}

ii) $\alpha U^\ast_\alpha u = u$ $m$-a.e., $\alpha > 0$.

\vspace{0.1cm}

iii) The function $ u $ is $\mathcal{U}$-invariant.

\vspace{0.1cm}

iv) The function $ u $ is measurable w.r.t. the $\sigma$-algebra of all $\mathcal{U}$-invariant sets.
\end{prop}
\begin{proof}
The equivalence i) $\Leftrightarrow$ iv) follows by \cite{Sc04}, Corollary 21. 
Also, from Proposition \ref{prop2.3} we have that $\mathcal{U}^\ast$ is $m$-recurrent, hence ii) $\Leftrightarrow$ iv).
The implication iv) $\Rightarrow$ iii) is obtained by approximating with simple functions, and
since iii) $\Rightarrow$ i) is clear, the proof is complete.
\end{proof}

However, we recall that if $m(E)<\infty$ then Proposition \ref{prop 2.17} is just a particular case of Theorem \ref{thm 2.19}.

The next proposition shows that in condition v) of Theorem \ref{thm 2.19} we can put inequality instead of equality.

\begin{prop} \label{prop 2.18} 
The following assertions are equivalent for a function $u \in L^p(E,m)$ such that $u^- \in L^1(E, m)$, $1 \leq p < \infty$.

i) $U_\alpha u = u U_\alpha 1$ and $ U_\alpha^\ast u = u U_\alpha^\ast 1$ $m$-a.e., $\alpha > 0$.

\vspace{0.1cm}

ii) $U_\alpha u \geq u U_\alpha 1$ and $U_\alpha^\ast u \geq u U_\alpha^\ast 1$ $m$-a.e., $\alpha > 0$.
\end{prop}
\begin{proof}
Since the implication i) $\Rightarrow$ ii) is trivial we prove only the converse. If condition ii) holds for $u$ then it holds for $u^+$ too and $U_\alpha (u^+)^p (U_\alpha 1)^{p-1} \geq (U_\alpha u^+)^p \geq (u^+)^p (U_\alpha 1)^p$, hence $U_\alpha (u^+)^p \geq (u^+)^p U_\alpha 1$ $m$-a.e. 
Because the same relations hold for $\mathcal{U}^\ast$ we have
\[
\int{U_\alpha (u^+)^p dm} = \int{(u^+)^p U_\alpha^\ast 1 dm} \leq \int{U_\alpha^\ast (u^+)^p dm} = \int{(u^+)^p U_\alpha 1 dm} \leq \int{U_\alpha (u^+)^p dm}
\] 
It follows that $u^+$ satisfies i) hence $U_\alpha u^- \leq u^- U_\alpha 1 \; {\rm and} \; U_\alpha^\ast u^- \leq u^- U_\alpha^\ast 1$ $m$-a.e. Then
\[
\int{U_\alpha u^- dm} = \int {u^- U_\alpha^\ast 1 dm} \leq \int{U_\alpha^\ast u^- dm} = \int{u^- U_\alpha 1 dm} \leq \int{U_\alpha u^- dm} 
\]
thus condition i) is also verified by $u^-$. 
\end{proof}

\begin{prop} \label{prop 2.20} 
If $\mathcal{U}$ is $m$-recurrent or $m$-symmetric then a set $A \in \mathcal{B}$ is $\mathcal{U}$-absorbing if and only if it is $\mathcal{U}$-invariant.  
\end{prop}
\begin{proof}
The symmetric case follows by Corollary \ref{coro 2.17}. 
Assume that $\mathcal{U}$ is $m$-recurrent. If $A$ is $\mathcal{U}$-absorbing then by Proposition \ref{prop2.4} 
there exists $B \in \mathcal{B}$ such that $B = A$ $m$-a.e. and $1_{E \setminus B} \in \mathcal{E}(\mathcal{U})$. 
By Proposition \ref{prop2.3} we have that $\alpha U_\alpha 1_{E \setminus B} = 1_{E \setminus B}$ $m$-a.e., $\alpha > 0$. 
To get that $A$ is $\mathcal{U}$-invariant we can simply apply Proposition \ref{prop 2.17} 
or notice that $ \alpha U_\alpha 1_A = \alpha U_\alpha 1 - \alpha U_\alpha 1_{E \setminus B} = 1_A $ $m$-a.e.,  
hence $E \setminus A$ is $\mathcal{U}$-absorbing and the implication follows by Corollary \ref{coro 2.17}. 
The converse is clear.
\end{proof}

\begin{coro} \label{coro 2.21} 
Consider the following assertions.

i) $\mathcal{U}$ is $m$-irreducible.

\vspace{0.1cm}

ii) Every $L^p(E, m)$-integrable $\mathcal{U}$-invariant function is constant, $1 \leq p < \infty$.

\vspace{0.2cm}

iii) Every bounded $\mathcal{U}$-invariant function is constant.

\vspace{0.1cm}

Then i) $\Rightarrow$ ii) $\Leftarrow$ iii).

If $m(E) < \infty$ then ii) $\Leftrightarrow$ iii). In addition, if $\mathcal{U}$ is $m$-symmetric then i) $\Leftrightarrow$ ii) $\Leftrightarrow$ iii).

If $\mathcal{U}$ is $m$-recurrent then i) $\Leftrightarrow$ iii). 
\end{coro}
\begin{proof}
i) $ \Rightarrow $ ii). If $\mathcal{U}$ is $m$-irreducible then the $\sigma$-algebra of all $\mathcal{U}$-invariant sets is trivial and assertion ii) follows by Theorem \ref{thm 2.19}.

The implication iii) $\Rightarrow$ ii) follows by Proposition \ref{prop 2.15} and Remark \ref{rem 2.16}, i). If $m(E) < \infty$ then the converse is clear. In addition, if $\mathcal{U}$ is $m$-symmetric then by Proposition \ref{prop 2.20} it follows that iii) $\Rightarrow$ i) hence all of the three assertions are equivalent.

Assume now that $\mathcal{U}$ is $m$-recurrent. If iii) holds and $A$ is $\mathcal{U}$-absorbing then by Proposition \ref{prop 2.20} it follows that $1_A$ is $\mathcal{U}$-invariant and therefore it is constant $m$-a.e. Thus $\mathcal{U}$ is $m$-irreducible. Conversely, assume that i) is satisfied. Then there are no non-trivial $\mathcal{U}$-invariant sets and assertion iii) is deduced from Proposition \ref{prop 2.17}.
\end{proof}

\section{Irreducibility and ergodicity of $L^p$-resolvents} 

In this section we study ergodic properties of a sub-Markovian resolvent of kernels under additional hypotheses such as transience and irreducible recurrence. 

Let $m$ be a $\sigma$-finite measure on $(E, \mathcal{B})$. Further we shall use the notation $(\cdot, \cdot)$ to express the duality between $L^p(E, m)$ and $L^{p'}(E, m)$ and $\| \cdot \|_p$ for the $L^p(E, m)$-norm; $p'$ is the exponential conjugate of $p$: $\frac{1}{p} + \frac{1}{p'} = 1$.

We say that a resolvent family $\mathcal{U} = (U_\alpha)_{\alpha > 0}$ of operators on $L^p(E,m)$, $1<p<\infty$ is {\it ergodic} if the strong limit $\lim\limits_{\alpha \to 0}\alpha U_{\alpha}u$ exists for all $u \in L^p(E, m)$. 
 
The next theorem states a convenient version for the present context of the classical result concerning Abel-ergodicity of a pseudo-resolvent family $(U_\alpha)_{\alpha > 0}$ of operators on a locally convex space; cf. \cite{Yo80}, Chapter VIII, Section 4 and \cite{WaEr93}. 
We drop the sub-Markov property for the moment and proceed with the more general condition that allows $\mathcal{U}$ to be uniformly bounded.

\begin{thm} \label{thm 3.1}  
Let $m$ be a $\sigma$-finite measure on $(E, \mathcal{B})$ and $\mathcal{U} = (U_\alpha)_{\alpha > 0}$ be a resolvent family of continuous linear operators on $L^p(E,m)$, $1<p<\infty$ such that $\|\alpha U_\alpha \|_p \leq M$ for all $\alpha > 0$ and for some positive constant $M < \infty$.
Then $\mathcal {U}$ is ergodic. More precisely, for one (hence for all) $\beta > 0$ and all $u \in L^p(E,m)$ there exists $u' \in {\rm Ker}(\mathcal{I} - \beta U_\beta)$ such that 
\[
\mathop{\lim}\limits_{\alpha \to 0} \| \alpha U_{\alpha}u - u'\|_p = 0.
\]
\end{thm} 
\begin{proof} Step I. We claim that for every $f \in L^p(E,m)$ there exists $\alpha_n \searrow 0$ such that $(\alpha_n U_{\alpha_n} f)_n$ is weakly convergent to some element from Ker$(\mathcal{I} - \beta U_\beta)$ and, as a consequence, that Ker$(\mathcal{I} - \beta U_\beta)$ separates Ker$(\mathcal{I} - \beta U^\ast_\beta)$, in the sense that if $v \in \mbox{Ker}(\mathcal{I} - \beta U^\ast_\beta)$ and $(u,v) = 0$ for all $u \in {\rm Ker}(\mathcal{I} - \beta U_\beta)$ then $v = 0$, where $ U^\ast_\beta$ is the adjoint operator of $U_\beta$ on $L^{p'}(E,m)$. 
To prove this, let $f \in L^p(E,m)$, $f' \in L^p(E,m)$ and $\alpha_n \downarrow 0$ s.t. $(g_n)_n := (\alpha_n U_{\alpha_n} f)_n$ is weakly convergent to $f'$. 
By passing to a subsequence we may assume that the sequence of Cesaro means $(\dfrac{1}{n}\mathop{\sum}\limits_{k =1}^{n}g_k)_n$ converges strongly to $f'$. 
Then
\[
\beta U_\beta(\dfrac{1}{n}\mathop{\sum}\limits_{k =1}^{n}g_k) = \dfrac{1}{n}\mathop{\sum}\limits_{k=1}^{n}\alpha_k \beta U_\beta U_{\alpha_k}f = \dfrac{1}{n}\mathop{\sum}\limits_{k=1}^{n}\dfrac{\alpha_k \beta}{\beta - \alpha_k}(U_{\alpha_k} f - U_\beta f) \mathop\rightarrow\limits_{n} f'.
\] 
It follows that $f' \in {\rm Ker}(\mathcal{I} - \beta U_\beta)$.

Assume now that $v \in \mbox{Ker}(\mathcal{I} - \alpha U^\ast_\beta)$ and $(u,v) = 0$ for all $u \in {\rm Ker}(\mathcal{I} - \beta U_\beta)$. If $f \in L^p(E, m)$ then by the first part of this proof there exists $f' \in {\rm Ker}(\mathcal{I} - \beta U_\beta)$ and $\alpha_n \downarrow 0$ s.t. $(\alpha_n U_{\alpha_n} f)_n$ is weakly convergent to $f'$. Then
\[
(f,v) = (f, \alpha_n U^\ast_{\alpha_n} v) = (\alpha_n U_{\alpha_n}f, v) \mathop{\rightarrow}\limits_n (f', v) = 0.
\]
Since $f$ was arbitrarily chosen it follows that $v = 0$ and Ker$(\mathcal{I} - \beta U_\beta)$ separates Ker$(\mathcal{I} - \beta U^\ast_\beta)$.

Step II. We show now that $(\alpha U_\alpha u)_{\alpha > 0}$ is strongly convergent for all $u \in L^p(E, m)$. Choose $\beta > 0$ and consider the subspace 
\[
G := \mbox{Ker}(\mathcal{I} - \beta U_\beta) \oplus \left\{ \beta U_\beta f - f : f \in L^p(E, m) \right\} 
\]
of $L^p(E,m)$ and take $v \in L^{p'}(E, m)$ s.t. $(u, v) = 0$ for all $u \in G$. Since $v$ is orthogonal  to each element of the form $\beta U_\beta f - f$, $f \in L^p(E, m)$ we have that $v \in {\rm Ker}(\mathcal{I} - \beta U^\ast_\beta)$ and by Step I it follows that $v = 0$. This means that $G$ is dense in $L^p(E, m)$ and because $(\alpha U_\alpha)_{\alpha > 0}$ is uniformly bounded it is enough to prove that $(\alpha U_\alpha u)_{\alpha > 0}$ is strongly convergent for $u \in G$ and in fact, for elements of the type $\beta U_\beta f - f$, $f \in L^p(E, m)$. By the resolvent equation we have
\[
\|\alpha U_\alpha (\beta U_\beta f - f)\|_p = \alpha \| \alpha U_\alpha U_\beta f - U_\beta f \|_p \leq \alpha \frac{M+M^2}{\beta}\|f\|_p \mathop{\rightarrow}\limits_{\alpha \to 0} 0
\] 
for all $f \in L^p(E, m)$ and Step II is complete.

It is clear now that Step I and Step II prove the theorem.
\end{proof}

\begin{rem}  
i). Recall that if $(T_t)_{t \geq 0}$ is an uniformly bounded strongly continuous semigroup on a reflexive Banach space it holds that $(\dfrac{1}{t}\int\limits_0^t{T(s)ds})_{t \geq 0}$ is ergodic to the projection on the null space of its generator, as $t$ tends to infinity. 
This property is known as {\it mean ergodicity} and its proof, for which we refer to \cite{EnNa99}, Chapter V, Theorem 4.5 and Example 4.7, follows the same lines as the one of Theorem \ref{thm 3.1}. 
We emphasize that if a strongly continuous semigroup is uniformly bounded then so is its corresponding resolvent, but the converse is not true in general, so the assumption on the boundedness of the resolvent is weaker. 
In the same time, it is natural that the resolvent gives more information about the semigroup rather than its integral means.

ii). If $(T_t)_{t\geq 0}$ is the transition function of a right Markov process on $(E,\mathcal{B})$ which is $m$-recurrent, then by \cite{Fi98}, Theorem 1.1, the following {\it quasi-sure} form of Theorem \ref{thm 3.1} holds:
let $\Sigma$ be the $\sigma$-algebra of all $m$-invariant sets and set $\mu := q \cdot m$ with $q>0$ and $m(q)=1$.  
Then, for every measurable functions $f$ and $g \geq 0$ from $L^1(m)$, there exists an $m$-polar set $B\in \mathcal{B}$ such that
$$
\lim\limits_{t \to \infty}\frac{\int\limits_0^tT_s f(x) ds}{\int\limits_0^tT_s g(x) ds} = \frac{\mu(f\slash q \mid \Sigma)}{\mu(g\slash q \mid \Sigma)}
$$
for all $x \in [Ug>0] \setminus B$;
for the corresponding statement in terms of resolvents see Theorem 6.1 from \cite{Fi98}.
This result is a generalization of its $m$-semipolar version proved in \cite{Fu74}, Theorem 3.1 for standard Markov processes in duality (see also \cite{Sh76} and \cite{Sh77}),
and, as a matter of fact, it is the quasi-sure refinement in continuous time of the well-known ergodic result of Chacon and Ornstein, \cite{ChOr60}. 
\end{rem}

In view of Theorem \ref{thm 2.19} we would like to give a better insight for Theorem \ref{thm 3.1} in the situation that $\mathcal{U}$ is a sub-Markovian resolvent of kernels on $(E, \mathcal{B})$ such that $\mathcal{E}(\mathcal{U}_\beta)$ is min-stable, contains the positive constant functions and generates $\mathcal{B}$, and $m$ a $\sigma$-finite sub-invariant measure, or equivalently, according to the discussion in the beginning of Section 2, that $\mathcal{U}$ (and hence $\mathcal{U}^\ast$) is a strongly continuous sub-Markovian resolvents of contractive operators on $L^p(E, m)$, $1 < p < \infty$. 
For every $1 < p < \infty$ we denote by $({\sf L}_p, D({\sf L}_p))$ the generator associated to $\mathcal{U}$ as a strongly continuous resolvent of operators on $L^p(E,m)$:
\[
D({\sf L}_p) = U_{\alpha}(L^p(E, m)), \; \alpha >0, \; {\sf L}_p(U_{\alpha}f) = \alpha U_{\alpha} f - f, \; f\in L^p(E, m).
\]  
The corresponding generator associated to $\mathcal{U}^\ast$ is denoted by $({\sf L}^{\ast}_{p'}, D({\sf L}^{\ast}))$.
We point out that the adjoint operator of ${\sf L}_p$ is ${\sf L}^\ast_{p'}$ and not ${\sf L}^\ast_{p}$.

The following corollary is a direct consequence of Theorem \ref{thm 2.19}.

\begin{coro} \label{prop 3} 
Let $1 \leq p < \infty$. Then the following assertions hold.

\vspace{0.2cm}

i) {\rm Ker}${\sf L}_p$ = {\rm Ker}${\sf L}^\ast_p$.

\vspace{0.2cm}

ii) {\rm Ker}${\sf L}_p \cap L^{p'}(E, m) \subset {\rm Ker}{\sf L}^\ast_{p'}$.

\vspace{0.2cm} 

iii) If $u \in {\rm Ker}{\sf L}_p$ then $u$ is $\mathcal{U}$-invariant. 

\vspace{0.2cm}

iv) If $\alpha U_\alpha 1 = 1$ or $\alpha U^\ast_\alpha 1 = 1$ $m$-a.e. then the converse of iii) is also true for any $u \in L^p(E, m)$.
\end{coro}

In potential theoretical terms, Corollary \ref{prop 3}, i), states that the harmonic and coharmonic functions belonging to $L^{p}$ coincide. 
In combination with Theorem \ref{thm 3.1}, it means that for functions $u \in L^p(E,m)$, the limit of $\alpha U_\alpha u$, $\alpha \searrow 0$, produces both harmonic and coharmonic functions.\\

From now on we consider the same framework as in Section 2, that is $\mathcal{U}=(U_\alpha)_{\alpha > 0}$ is a sub-Markovian resolvent of kernels on $(E, \mathcal{B})$
and $m$ is a $\sigma$-finite sub-invariant measure with respect to $\mathcal{U}$. In particular, $\mathcal{U}$ becomes a resolvent family of contractive operators on $L^p(E,m)$ for all $1 < p < \infty$, hence $\mathcal{U}$ is ergodic in the sense of Theorem \ref{thm 3.1}. 

In the next two propositions we exploit the ergodic property of $\mathcal{U}$, involving the additional properties of $m$-transience and $m$-irreducible recurrence, respectively. 

\begin{prop} \label{prop 3.4} 
Consider the following assertions.

\vspace{0.2cm}

i) $\mathcal{U}$ is $m$-transient.

\vspace{0.2cm}

ii) If $u \in L^p$, $1 < p < \infty$, and $\alpha U_\alpha u = u$ then $u = 0$ $m$-a.e.

\vspace{0.2cm}

iii) For all $u \in L^p(E, m)$, $1 < p < \infty$ we have
\[
\lim\limits_{\alpha \to 0} \| \alpha U_\alpha u \|_p  = 0.
\]

Then i) $\Rightarrow$ ii) $\Leftrightarrow$ iii).

If $m(E) < \infty$ and $\mathcal{U}$ is $m$-irreducible then i), ii), and iii) are equivalent.
\end{prop}

\begin{proof}
i) $\Rightarrow$ iii). We may assume that $u$ is positive. 
If $u \in p\mathcal{B} \cap L^p(E, m) \cap L^1(E, m)$ then by Theorem \ref{thm 3.1} there exists $ u' \in L^p(E, m)$ such that $\alpha U_\alpha u' = u'$ and $\lim\limits_{\alpha \to 0} \| \alpha U_\alpha u - u' \|_p = 0 $. 
By Proposition \ref{prop2.1} we have $\alpha U_\alpha u \leq \alpha U u \mathop{\longrightarrow}\limits_{\alpha \to 0} 0$ $m$-a.e.,
therefore $u' = 0$ $m$-a.e. and $\lim\limits_{\alpha \to 0} \alpha U_\alpha u = 0$ in $ L^p(E, m) $. 
Let now $ u, u' \in p\mathcal{B} \cap L^p(E, m) $ and $(u_n)_n \subset p\mathcal{B} \cap L^p(E, m) \cap L^1(E, m)$ such that $\lim\limits_{n \to \infty} \| u_n - u \|_p = 0$ and $\lim\limits_{\alpha \to 0} \|\alpha U_\alpha u - u' \|_p = 0$ (cf. Theorem \ref{thm 3.1}). 
Then
\[
\|u'\|_p = \lim\limits_{\alpha \to 0} \| \alpha U_\alpha u \|_p \leq \lim\limits_{\alpha \to 0} ( \| \alpha U_\alpha (u - u_n) \|_p + \| \alpha U_\alpha u_n \|_p) \leq 
\]
\[
\leq \| u - u_n \|_p + \lim\limits_{\alpha \to 0} \| \alpha U_\alpha u_n \|_p = \| u - u_n \|_p \mathop{\longrightarrow}\limits_{n} 0
\]
hence $u' = 0$ $m$-a.e.

iii) $\Rightarrow$ ii). If $u \in L^p(E,m)$ such that $\alpha U_\alpha u = u$ then by iii) we have $u = 0$ $m$-a.e.

ii) $\Rightarrow$ iii). If $u \in  L^p(E, m)$ by Theorem \ref{thm 3.1} there exists $u' \in L^p(E,m)$ such that $\alpha U_\alpha u' = 0$ and $\lim\limits_{\alpha \to 0} \| \alpha U_\alpha u - u' \|_p = 0$. By ii) we have $u' = 0$ $m$-a.e.

Assume now that $m(E) < \infty$ and $\mathcal{U}$ is $m$-irreducible. 
Since $1 \in L^p(E,m)$, if iii) holds then we have that $\alpha U_\alpha 1 $ converges pointwise to $0$ as $\alpha$ goes to $0$ hence $\mathcal{U}$ is not $m$-recurrent. 
By Proposition \ref{prop 2.5} it follows that $\mathcal{U}$ is $m$-transient.
\end{proof}

\begin{prop} \label{prop 3.5} 
Consider the following conditions.

i) $\mathcal{U}$ is $m$-irreducible and $m$-recurrent.

\vspace{0.1cm}

ii) If $u \in L^p(E,m)$ and $\alpha U_\alpha u = u$ then $u$ is constant.

\vspace{0.1cm}

iii) For all $u \in L^p(E, m) $ we have
\[
\mathop{\lim}\limits_{\alpha \to 0}\| \alpha U_{\alpha}u - c_u \|_p = 0,
\]
where $c_u$ is the constant defined by
\[
c_u := \left\{\begin{array}{ll}
\dfrac{\int u dm}{m(E)} &, \; {\rm if} \; m(E) < \infty \\
0 &, \; {\rm if} \; m(E) = +\infty.
\end{array}
\right.
\]

\vspace{0.1cm}

Then i) $\Rightarrow$ iii) $\Rightarrow$ ii).

If $m(E) = \infty$ then ii) $\Leftrightarrow$ iii).

If $m(E) < \infty$ then i) $\Leftrightarrow$ iii). If in addition $\mathcal{U}$ is $m$-recurrent then i), ii), and iii) are equivalent.
\end{prop}
\begin{proof}
i) $\Rightarrow$ iii). Let $u \in L^p(E, m)$. 
From Theorem \ref{thm 3.1} there exists $ u' \in L^p(E,m)$ such that $\alpha U_\alpha u' = u'$ and $\lim\limits_{\alpha \to 0} \| \alpha U_\alpha u - u' \|_p = 0 $. Furthermore, by Proposition \ref{prop 2.9} we have that $u'$ is constant. 
Clearly $u' = 0$ if $m(E) = \infty$. 
If $m(E) < \infty$ then $\int u' dm = \mathop{\lim}\limits_{\alpha \to 0}(\alpha U_{\alpha}u, 1) = \mathop{\lim}\limits_{\alpha \to 0}(u, \alpha U_{\alpha}^{\ast}1) = \int u dm$ and so $u' = \dfrac{\int u dm}{m(E)}$.

The implication iii) $\Rightarrow$ ii) is clear.

If $m(E) = \infty$ and ii) holds then for $u \in L^p(E, m)$ and $u'$ provided by Theorem \ref{thm 3.1} we have that $u'$ is constant and therefore $u' = 0$ and iii) is satisfied.

Assume now that $m(E) < \infty$. 
Under assertion iii), if $u$ is a bounded excessive function then $\int u dm \geq \int \alpha U_\alpha u dm \mathop\to\limits_{\alpha \to 0} \int u dm$, 
hence $\alpha U_\alpha u = u$ $m$-a.e. and in fact $u = c_u$ $m$-a.e. 
By Proposition \ref{prop 2.9} we get that $\mathcal{U}$ is $m$-irreducible recurrent.
If in addition $\mathcal{U}$ is $m$-recurrent and ii) holds, it follows once again that every bounded excessive function is constant and by Proposition \ref{prop 2.9} we conclude that $\mathcal{U}$ is $m$-irreducible and assertions i), ii), and iii) are equivalent.  
\end{proof}

\begin{rem} \label{rem 3.6} 
a) If $m(E) < \infty$ then the $L^p$-ergodicity in the assertions iii) of Proposition \ref{prop 3.4} and respectively of Proposition \ref{prop 3.5} for $p > 1$ implies also the $L^1$-ergodicity. 
This follows easily by the density of $L^1(E, m) \cap L^p(E, m)$ in $L^1(E, m)$.

b) If $\mathcal{U}$ is the resolvent of an $m$-symmetric right process $X$, then by \cite{FuOsTa11}, Theorem 4.7.3, if $\mathcal{U}$ is $m$-irreducible and $m$-recurrent then for all Borel measurable and $m$-integrable functions $u$ we have $P_m$-a.s. and $P_x$-a.s. for q.e. $x \in E$ that
\[
\lim\limits_{t \to \infty}\frac{1}{t}\int\limits_0^t u(X_s)ds = c_u.
\]
This pathwise ergodicity is entailed by a corresponding ergodicity in terms of shift invariance for which we refer to Theorem 4.7.2.
\end{rem}

\section{Extremality of invariant measures} As in the previous sections, we assume that $\mathcal{U} = (U_\alpha)_{\alpha > 0}$ is merely a sub-Markovian resolvent of kernels on $(E, \mathcal{B})$. 
Let $\mathcal{I}$ be the set of all $\mathcal{U}$-invariant probability measures, i.e. 
$$
\mathcal{I} : = \{\mu : \mu  \; \mbox{ is a}  \; \mbox{ probability measure such that } \; \mu \circ \alpha U_\alpha = \mu, \; \alpha > 0 \}. 
$$
We fix an $\mathcal{U}$-sub-invariant probability measure $m$ and denote by $\mathcal{I}_{m, ac}$ the subset of $\mathcal{I}$ which consists of all absolutely continuous measures with respect to $m$. 
Also, let $\mathcal{U}^\ast$ be the adjoint resolvent of $\mathcal{U}$ with respect to $m$ (cf. (2.1)).

\begin{lem} \label{lem 2.24} 
The following assertions are equivalent for a probability measure $\mu$.

i) $\mu \in \mathcal{I}_{m,  ac}$.

\vspace{0.1cm}

ii) There exists a function $u \in p\mathcal{B}$ such that $\alpha U_\alpha u = u$ $m$-a.e. for all $\alpha > 0$, $\mu = u \cdot m$, and $m(u) = 1$. 
In particular, the function $u$ is $\mathcal{U}$-invariant  (or equivalently,  $\mathcal{U}^\ast$-invariant).
\end{lem}

\begin{proof}
i) $\Rightarrow$ ii). Let $u \in p\mathcal{B} \cap L^1(E, m)$ such that $\mu = u \cdot m \in \mathcal{I}_{m, ac}$. Then for every $f \in bp\mathcal{B}$ we have
\begin{equation} \label{eq 4.1}
\int{f \alpha U_\alpha^\ast u dm} = \int{u \alpha U_\alpha f dm = \int{fu dm}}
\end{equation}
hence $\alpha U_\alpha^\ast u = u$, $\alpha > 0$. By Theorem \ref{thm 2.19} we conclude that $u$ is $\mathcal{U}$-invariant.

The implication ii) $\Rightarrow$ i) follows by Proposition \ref{thm 2.19} and relation (\ref{eq 4.1}).
\end{proof}

Let $\mathcal{G}^m$ be the set of all probability measures $\mu$ on $(E, \mathcal{B})$ of the form $\mu = u\cdot m$, where $u$ is $\mathcal{U}$-invariant (with respect to $m$).

\begin{rem} \label{rem 2.22} 
i) Because the $\mathcal{U}$-invariant functions are $\mathcal{U}^{\ast}$-excessive, it follows that $\mathcal{G}^m$ is a set of sub-invariant measures.

ii) $\mathcal{G}^{m}$ is a non-empty convex set; $m \in \mathcal{G}_{m}$.

iii) $\mathcal{U}$ is $m$-recurrent if and only if $m \in \mathcal{I}$.

iv) We have the inclusion $\mathcal{I}_{m, ac} \subset \mathcal{G}^m $. 
If $\mathcal{U}$ is $m$-recurrent then $\mathcal{G}^m = \mathcal{I}_{m, ac}$.
\end{rem}

\begin{thm} \label{thm 2.26} 
Consider the following assertions.

i) $\mathcal{U}$ is $m$-irreducible.

\vspace{0.1cm}

ii) $\mathcal{G}^{m} = \{m\}$.

\vspace{0.1cm}

iii) The measure $m$ is extremal in $\mathcal{G}^{m}$.

\vspace{0.1cm}

iv) The measure $m$ is extremal in $\mathcal{I}$.

\vspace{0.1cm}

Then i) $\Rightarrow$ ii) $\Leftrightarrow$ iii). 

If $\mathcal{U}$ is $m$-symmetric then assertions i) - iii) are equivalent.

If $\mathcal{U}$ is $m$-recurrent then assertions i) - iv) are equivalent.
\end{thm} 

\begin{proof}
The implication i) $\Rightarrow$ ii) follows by Corollary \ref{coro 2.21}.

ii) $\Leftrightarrow$ iii). Clearly, ii) implies iii). Assume that iii) holds and let $\mu \in \mathcal{G}^m$, $\mu = u\cdot m$. 
If $u \leq 1$ or $u \geq 1$, since $\int u \wedge 1 dm = 1$, we get that $\mu = m$. Assume that $\mu \neq m$. 
Consequently we have $0 < \int u \wedge 1 dm < 1$, hence if we put $\alpha := \int u \wedge 1 dm$ then $\alpha \in (0,1)$. 
By Proposition \ref{prop 2.15} it follows that $u \wedge 1$ is also $\mathcal{U}$-invariant function. 
Therefore the measures $\mu_1 := \dfrac{u \wedge 1}{\alpha}\cdot m$ and $\mu_2 := \dfrac{1 - u\wedge 1}{1-\alpha}\cdot m$ belong to $\mathcal{G}^m$ and clearly we have $m = \alpha \mu_1 + (1-\alpha)\mu_2$. 
The measure $m$ being extremal in $\mathcal{G}^m$, we get the contradictory equality $\mu_1 = m$ and therefore ii) holds.

Assume now that $\mathcal{U}$ is $m$-symmetric. 
Clearly, is enough to show that ii) $\Rightarrow$ i), so let $A \subset E$ be an $\mathcal{U}$-absorbing set. 
By Proposition \ref{prop 2.20} we get that the function $1_A$ is $\mathcal{U}$-invariant. 
If $m(A) > 0$ then the measure $\dfrac{1_A}{m(A)}\cdot m$ belongs to $\mathcal{G}^m = \{m\}$ so $m(E \setminus A) = 0$, hence $\mathcal{U}$ is $m$-irreducible.

Le us consider the last case when $\mathcal{U}$ is $m$-recurrent. 
iii) $\Rightarrow$ iv).  
If $m = \alpha m_1 + ( 1 - \alpha ) m_2$ with $m_1$, $m_2 \in \mathcal{I}$ and $\alpha \in (0, 1)$ then by Lemma \ref{lem 2.24} and Remark \ref{rem 2.22}, iv) we have that $m_1 \in \mathcal{I}_{m, ac} = \mathcal{G}^m$, hence $m_1 = m$ and $m$ is extremal in $\mathcal{I}$.

iv) $\Rightarrow$ i). First we notice that under condition iv), from Remark \ref{rem 2.22}, iv) it follows that iii) and hence ii) hold. 
Let now $A \subset E$ be an $\mathcal{U}$-absorbing set. 
By Proposition \ref{prop 2.20} we get that the function $1_A$ is $\mathcal{U}$-invariant. 
If $m(A) > 0$ then the measure $\dfrac{1_A}{m(A)}\cdot m$ belongs to $\mathcal{G}^m = \{m\}$ so $m(E \setminus A) = 0$. 
In conclusion we obtain that $\mathcal{U}$ is $m$-irreducible.
\end{proof}

As an application of Theorem \ref{thm 2.26}, we end this sections with the following known result (cf. e.g. \cite{DaZa96}, Proposition 3.2.5; strongly continuous semigroups) on the singularity of extremal invariant measures. We make the remark that, in contrast with the previous work, we drop the strong continuity assumption.
The key ingredient is the ergodicity of $\mathcal{U}$ with respect to an extremal measure. 

\begin{prop} \label{prop 3.8} 
If $\mu$ and $\nu$ are extremal measures in $\mathcal{I}$ such that $\mu \neq \nu$ then $\mu$ and $\nu$ are singular.
\end{prop}

\begin{proof}
Let $A \in \mathcal{B}$ such that $\mu (A) \neq \nu (A)$. 
By Remark \ref{rem 2.22}, iii), Theorem \ref{thm 2.26}, and Proposition \ref{prop 3.5} there exists a sequence $(\alpha_n)_{n \geq 1}$ decreasing to $0$ such that 
\[
\lim\limits_n \alpha_n U_{\alpha_n} 1_A = \mu (A), \; \mu \mbox{-a.e. and} \; \lim\limits_n \alpha_n U_{\alpha_n} 1_A = \nu (A), \; \nu \mbox{-a.e.}
\]
If we set 
\[
\Gamma_1 : = \left\{ x \in E : \lim\limits_n \alpha_n U_{\alpha_n} 1_A(x) = \mu (A) \right\},\
\Gamma_2 : = \left\{ x \in E : \lim\limits_n \alpha_n U_{\alpha_n} 1_A(x) = \nu (A) \right\},
\]
then $\Gamma_1 \cap \Gamma_2 = \emptyset$ and $\mu(A) = \nu(A) = 1$. Therefore $\mu$ and $\nu$ are singular.
\end{proof}

\section{Irreducibility  of (non-symmetric) Dirichlet forms}
In this section we assume that $\mathcal{U}$ and $\mathcal{U}^{\ast}$ are the resolvent and respectively the co-resolvent of a (non-symmetric) Dirichlet form $(\mathcal{E}, D(\mathcal{E}))$ on $L^2(E, m)$, i.e.,
\[
U_{\alpha}(L^2(E, m)) \subset D(\mathcal{E}), \; U_{\alpha}^{\ast}(L^2(E, m)) \subset D(\mathcal{E})
\]
and
\[
\mathcal{E}_{\alpha}(U_{\alpha}f, u) = \mathcal{E}_{\alpha}(u, U_{\alpha}^{\ast}f) = (f, u)_{L^2(E, m)}
\]
for all $\alpha > 0$, $f \in L^2(E, m)$ and $u \in D(\mathcal{E})$, where $\mathcal{E}_{\alpha}: = \mathcal{E} + \alpha(\cdot, \cdot)_{L^2(E, m)}$;
for the definition of the  (non-symmetric) Dirichlet form see \cite{MaRo92}, Definition 4.5.
Recall that $\mathcal{U}$ and $\mathcal{U}^\ast$ become (uniquely) strongly continuous sub-Markovian resolvents of contractive operators on $L^2(E, m)$  and $m$ is a sub-invariant measure (cf. \cite{MaRo92}, Theorem 2.8 and Theorem 4.4; see also Chapter II, Section 5).
 
According to the discussion at the beginning of Section 2, we can assume that $\mathcal{U}$ and $\mathcal{U}^\ast$ are sub-Markovian resolvents of kernels in weak duality with respect to $m$. In particular, all notions  that are related to $\mathcal{U}$ depend implicitly on the fixed measure $m$.

We suppose that $\mathcal{E}$ satisfies the {\it (strong) sector condition}, that is there exists a constant $k > 0$ such that
\[
|\mathcal{E}(u, v)| \leq k\mathcal{E}(u, u)^{\frac{1}{2}} \mathcal{E}(v,v)^{\frac{1}{2}}
\]
for all $u,v \in D(\mathcal{E})$.

We denote by $({\sf L}, D({\sf L}))$ (resp. $({\sf L}^{\ast}, D({\sf L}^{\ast}))$) the generator (resp. co-generator) of the form $(\mathcal{E}, D(\mathcal{E}))$,
\[
D({\sf L}) := U_{\alpha}(L^2(E, m)), \; {\sf L}(U_{\alpha}f) := \alpha U_{\alpha} f - f, \; f\in L^2(E, m)
\]
and recall that 
$\mathcal{E}(u,v) = -({\sf L}u,v)_{L^2(E, m)}$ for all  $u \in D({\sf L})$ and  $v \in D(\mathcal{E}).$

For the reader convenience we restate and prove the next well-known characterization of zero-energy elements.

\begin{lem} \label{lem 5.1} 
The following assertions are equivalent for $u \in L^2(E, m)$.

i) $u \in D(\mathcal{E})$ and $\mathcal{E}(u,u) = 0$.

\vspace{0.1cm}

ii) $u \in D({\sf L})$ and ${\sf L}u = 0$.

\vspace{0.1cm}

iii) $\alpha U_{\alpha}u = u$ for one (or equivalent for all) $\alpha > 0$.
\end{lem}
\begin{proof}
i) $\Rightarrow$ iii). Let $u \in D(\mathcal{E})$ with $\mathcal{E}(u,u) = 0$ and $f \in L^2(E, m)$, then by the sector condition we get $\mathcal{E}(u, U_{\alpha}^{\ast} f) = 0$,
\[
(u - \alpha U_{\alpha}u, f)_{L^2(E, m)} = \mathcal{E}_{\alpha}(u, U_{\alpha}^{\ast}f) - \alpha(u, U_{\alpha}^{\ast}f) = \mathcal{E}(u, U_{\alpha}^{\ast}f) = 0,
\]
hence $\alpha U_{\alpha} u = u$.

The implication iii) $\Rightarrow$ ii) is clear by  the definition of $({\sf L}, D({\sf L}))$ and ii) $\Rightarrow$ i) follows since $\mathcal{E}(u,u) = -({\sf L}u, u)$ if $u \in D({\sf L})$.
\end{proof}

\begin{prop} \label{prop 5.2} 
The following assertions are equivalent for $u \in L^2(E, m) \cap L^{\infty}(E, m)$.

i) $u$ is $\mathcal{U}$-invariant (with respect to $m$).

\vspace{0.1cm}

ii) If $v \in D({\sf L})$ then $uv \in D({\sf L})$ and ${\sf L}(uv) = u{\sf L}v$.

\vspace{0.1cm}

iii) If $v,w \in D(\mathcal{E})$ then $uv \in D(\mathcal{E})$ and $\mathcal{E}(uv, w) = \mathcal{E}(v, uw)$.
\end{prop}

\begin{proof}
i) $\Rightarrow$ ii). Let $\alpha > 0$ and $v = U_{\alpha}f$, $f \in L^2(E, m)$. 
If $u$ is $\mathcal{U}$-invariant then $uv = U_{\alpha}(uf)$. 
Therefore $uv \in D({\sf L})$ and
\[
{\sf L}(uv) = \alpha U_{\alpha}(uf) - uf = u(\alpha U_{\alpha}f - f) = u{\sf L}v.
\]

ii) $\Rightarrow$ i). Let $f \in L^2(E, m)$, $v = U_{\alpha}f$. 
Then by ii) there exists $g \in L^2(E, m)$ such that $uv = U_{\alpha}g$ and from ${\sf L}(uv) = u{\sf L}v$ we get
\[
{\sf L}(uv) = \alpha U_{\alpha}g - g = \alpha uv - g, \; u{\sf L}v = u(\alpha U_{\alpha}f - f) = \alpha uv - uf.
\]
Hence $g = uf$ and thus $u U_{\alpha}f = U_{\alpha}(uf)$, i.e. $u$ is $\mathcal{U}$-invariant.

i) $\Rightarrow$ iii). Let $v,w \in D(\mathcal{E})$ and $u$ be $\mathcal{U}$-invariant. Then
\[
\mathop{\sup}\limits_{\alpha} \mathcal{E}^{\alpha}(uv, uv) = \mathop{\sup}\limits_{\alpha}\int \alpha uv(uv - \alpha U_{\alpha} uv) dm = 
\mathop{\sup}\limits_{\alpha}\int u^2\alpha v(v -\alpha U_{\alpha}v) dm \leq
\]
\[
\leq \| u \|^2_{\infty}\mathop{\sup}\limits_{\alpha}\int \alpha v(v - \alpha U_{\alpha}v) dm \leq \|u\|^2_{\infty}k^2\mathcal{E}(v,v) < \infty.
\]
We deduce that $uv \in D(\mathcal{E})$ and therefore
\[
\mathcal{E}(uv, w) = \mathop{\lim}\limits_{\alpha \to \infty}\mathcal{E}^{\alpha}(uv, w) = \mathop{\lim}\limits_{\alpha \to \infty}\int \alpha w(uv - \alpha U_{\alpha}(uv)) dm =
\]
\[
= \mathop{\lim}\limits_{\alpha \to \infty}\int\alpha uw(v -\alpha U_{\alpha}v) dm = \mathcal{E}(v, uw).
\]

iii) $\Rightarrow$ i). Let $v = U_{\alpha} f$, $w = U_{\alpha}^{\ast}g$ with $f,g \in bp\mathcal{B}\cap L^2(E, m)$. By hypothesis iii) we have $uv, uw \in D(\mathcal{E})$ and
\[
\int uf U_{\alpha}^{\ast}g dm = \mathcal{E}_{\alpha}(v, uw) = \mathcal{E}_{\alpha}(uv, w) = \int gu U_{\alpha}f dm.
\]
According with Proposition \ref{prop 2.16} we conclude that $u$ is $\mathcal{U}$-invariant.
\end{proof}

\begin{rem}
By Theorem \ref{thm 2.19}, if $u \in L^2(E, m) \cap L^{\infty}(E, m)$ satisfies any (and hence all) of the conditions in Lemma \ref{lem 5.1} then $u$ also satisfies the ones in Proposition \ref{prop 5.2}. If $\mathcal{E}$ is Markovian then the converse is also true. 
\end{rem}

As in \cite{Fu07}, Definition  4.1, the Dirichlet form $\mathcal{E}$ is called {\it recurrent} (resp. {\it irreducible}) if the associated resolvent $\mathcal{U}$ is $m$-recurrent (resp. $m$-irreducible). Let $(\widetilde{\mathcal{E}}, D(\widetilde{\mathcal{E}}))$ be the symmetric part of $\mathcal{E}$,
\[
\widetilde{\mathcal{E}}(u,v):= \frac{1}{2}(\mathcal{E}(u,v) + \mathcal{E}(v,u)) \;  \mbox{for all} \; u,v \in D(\widetilde{\mathcal{E}}) := D(\mathcal{E}).
\]

\begin{coro} \label{coro 5.3} 
Assume that $m(E) < \infty$.

i) The following assertions are equivalent.

i.1) $\mathcal{E}$ is recurrent.

\vspace{0.1cm}

i.2) $\widetilde{\mathcal{E}}$ is recurrent.

\vspace{0.1cm}

i.3) $1 \in D(\mathcal{E})$ and $\mathcal{E}(1,1) = 0$.

\vspace{0.1cm}

ii) If $\mathcal{E}$ is recurrent then the following assertions are equivalent.

\vspace{0.1cm}

ii.1) $\mathcal{E}$ is irreducible.

\vspace{0.1cm}

ii.2) $\widetilde{\mathcal{E}}$ is irreducible.

\vspace{0.1cm}

ii.3) If $u \in D(\mathcal{E})$ and $\mathcal{E}(u,u)=0$ then $u$ is constant.

\vspace{0.1cm}

ii.4) If $u \in D({\sf L})$ and ${\sf L}u = 0$ then $u$ is constant.

\vspace{0.2cm}

ii.5) If $u \in D({\sf L})$ such that for all $v \in D({\sf L})$ we have that $uv \in D({\sf L})$ and ${\sf L}(uv) = u{\sf L}v$ then $u$ is constant.

\vspace{0.1cm}

ii.6) $\int{(\alpha U_\alpha u - \frac{1}{m(E)}\int{u dm})^2 dm} \mathop{\longrightarrow}\limits_{\alpha \to 0} 0$ for all $u \in L^2(E, m)$.
\end{coro}
\begin{proof}
The equivalence i.1) $\Leftrightarrow$ i.3) follows by Lemma \ref{lem 5.1} and Corollary \ref{coro 2.4}. 
This also shows the equivalence with i.2) of the other two assertions. 
The fact that assertions ii.1) - ii.6) are equivalent follows by a simple combination of Lemma \ref{lem 5.1}, 
Proposition \ref{prop 5.2}, Corollary \ref{coro 2.21}, and Proposition \ref{prop 3.5}. 
\end{proof}

\noindent
{\bf Extremality of Gibbs states.} From now on our framework is the one considered in \cite{AlKoRo97a}. 
All measures which appear are probability measures on a locally convex topological vector space $E$ and its borel $\sigma$-algebra $\mathcal{B}$. 
The set $\mathcal{F}C_b^\infty$ of all {\it finitely based smooth bounded functions on $E$} is defined as
\[
\mathcal{F}C_b^\infty : = \{ f(l_1, \ldots, l_n) : n \in \mathbb{N}, f \in C_b^\infty(\mathbb{R}^n), l_1, \ldots, l_n \in E'\},
\]
where $E'$ is the topological dual space of $E$. For $K \subset E$ and $(b_k)_{k \in K}$ 
a family of $\mathcal{B}$-measurable functions, we denote by $\mathcal{G}^b$ the set of all probability measures $m$ on $E$ such that $b_k \in L^2(E, m)$ and 
\begin{equation} \label{eq 5.1}
\int{\frac{\partial u}{\partial k} dm} = - \int{u b_k dm}, 
\end{equation}
for all $u \in \mathcal{F}C_b^\infty$ and $k \in K$. Elements in $\mathcal{G}^b$ are called {\it Gibbs states associated with b}. 
We fix $K$, $b$, and $m \in \mathcal{G}^b$, and we consider the corresponding Dirichlet form 
$(\mathcal{E}_{m, k}, D(\mathcal{E}_{m, k}))$ defined as the closure on $L^2 (E, m)$ of
\[
\mathcal{E}_{m, k}(u, v) = \int{\frac{\partial u}{\partial k}\frac{\partial v}{\partial k} dm}; \; \; \; \; \; \; u,\; v \in \mathcal{F}C_b^\infty.
\]

Hereinafter we assume that $K$ is countable and 
\[
\mathop{\sum\limits_{k \in K}} |l(k)|^2 < \infty \; \; \; \; \; \; {\rm for \; all} \; l \in E'.
\]
Then we can define the Dirichlet form $(\mathcal{E}_m, D(\mathcal{E}_m))$ by setting
\[
D(\mathcal{E}_m) : = \{ u \in \bigcap\limits_{k \in K} D(\mathcal{E}_{m, k}) : \sum\limits_{k \in K} \mathcal{E}_{m, k}(u, u) < \infty \}
\]
and
\[
\mathcal{E}_m(u, v) : = \frac{1}{2} \sum\limits_{k \in K} \mathcal{E}_{m, k}(u, v),  \; \; \; \; u, \; v \in D(\mathcal{E}_m).
\]
For more details on the definition and closability  of the forms introduced above see \cite{AlKoRo97a} and the references therein.

Further, we denote by $\mathcal{U} = (U_\alpha)_{\alpha > 0}$ the resolvent of kernels associated to $(\mathcal{E}_m, D(\mathcal{E}_m))$ and, as in Section 3, let $\mathcal{I}_{m,ac}$ be the set of all $\mathcal{U}$-invariant probability  measures which are absolutely continuous with respect to $m$.

\begin{rem} 
Since $1 \in D(\mathcal{E}_m)$ and $\mathcal{E}_m(1,1) = 0$, by Corollary \ref{coro 5.3} we have that $\mathcal{E}_m$ is recurrent.
\end{rem}

Let $\mathcal{G}^b_{m, ac}$ be the set of all probability measures from  $\mathcal{G}^b$ which are absolutely continuous with respect to $m$. 
We give now  a characterization of $\mathcal{G}^b_{m, ac}$ in terms of excessive functions.

\begin{thm} \label{thm 5.5} 
The following assertions are equivalent for $\rho \in p\mathcal{B} \cap L^1(E, m)$ such that $m(\rho) = 1$.

i) There exists an $\mathcal{U}$-excessive function which is an $m$-version of $\rho$.

\vspace{0.1cm}

ii) The measure $\rho \cdot m$ belongs to $ \mathcal{G}^b$.

\vspace{0.1cm}

Consequently, $\mathcal{I}_{m, ac} = \mathcal{G}^b_{m, ac}$.
\end{thm}
\begin{proof} 
i) $\Rightarrow$ ii). Without loss of generality we may assume that $\rho \in  \mathcal{E}(\mathcal{U})$. 
Suppose first that  $\rho \in bp(\mathcal{B}) \cap \mathcal{E}(\mathcal{U})$. 
By Proposition \ref{prop2.3} it follows that $\alpha U_\alpha \rho = \rho$ for all $\alpha > 0$, hence $\rho \in D(\mathcal{E}_{m})$. 
From Lemma \ref{lem 5.1}  we have that $\mathcal{E}_m(\rho, \rho) = 0$ and by the chain rule ( see \cite{AlKoRo97a}, Remark 1.1, iii)) 
and \cite{AlKuRo90}, Theorem 2.5, we conclude that $\rho \cdot m \in \mathcal{G}^b$. 
If $\rho \in \mathcal{E}(\mathcal{U})$ then $(\inf(\rho, n))_{n} \subset bp\mathcal{B} \cap \mathcal{E}(\mathcal{U})$ 
converges pointwise to $\rho$ and by the dominated convergence theorem  applied in relation (\ref{eq 5.1}) we get that $\rho \cdot m \in \mathcal{G}^b$.

ii) $\Rightarrow$ i).  If  $\rho \cdot m \in \mathcal{G}^b$ then by \cite{BoRo95}, Lemma 6.14 it follows that $\sqrt{\rho} \in D(\mathcal{E}_m)$ and $\mathcal{E}_m(\sqrt{\rho}, \sqrt{\rho}) = 0$. 
By Lemma \ref{lem 5.1} and Theorem \ref{thm 2.19} it follows that $\sqrt{\rho}$ is $\mathcal{U}$-invariant. 
Since
\[
\rho = \rho \alpha U_\alpha 1 = \alpha \sqrt{\rho} U_\alpha \sqrt{\rho} = \alpha U_\alpha \rho \quad m \mbox{-a.e.}
\]
by Proposition \ref{prop 2.2} we get that condition i) is satisfied.
\end{proof} 

Now, Theorem \ref{thm 5.5} places us in the context of Theorem \ref{thm 2.26} and we obtain:

\begin{thm} \label{thm 5.6} 
The following assertions are equivalent.

i) The form $\mathcal{E}_m$ is irreducible.

\vspace{0.1cm}

ii) $\mathcal{G}^b_{m, ac} = \{m\}$.

\vspace{0.1cm}

iii) The measure $m$ is extremal in $\mathcal{G}^b_{m, ac}$.

\vspace{0.1cm}

iv) The measure $m$ is extremal in $\mathcal{G}^b$.
\end{thm}
\begin{proof}
The fact that i), ii), and iii) are equivalent follows by Theorem \ref{thm 5.5} and Theorem \ref{thm 2.26}. 
Clearly iv) implies iii). 
If iii) holds and $m = \alpha m_1 + (1 - \alpha) m_2$ 
with $\alpha \in (0, 1)$ and $m_1, m_2 \in \mathcal{G}^b$, 
then $m_1$ and $m_2$ belong to $\mathcal{G}^b_{m, ac}$, hence $m=m_1=m_2$.
\end{proof}

In the spirit of \cite{RoWa01}, further connections between irreducibility, extremality of Gibbs measures, and functional inequalities can be studied. 
In this sense, we would also like to refer to the classical work of \cite{StZe92a}, \cite{StZe92b}, and \cite{StZe92c}.

\vspace{1mm}

\noindent
{\bf Example} (The non-symmetric case). 
Assume now that there exists a separable real Hilbert space $(H, \langle  , \rangle_H)$ densely and continuously embedded into $E$ and $K$ is an orthonormal basis of $H$. 
By the chain rule, for every $u \in \mathcal{F}C_b^\infty$ and $z \in E$ fixed, $h \mapsto \frac{\partial u}{\partial h}(z)$ is a continuous linear functional on $H$, hence $\nabla u(z) \in H$ is uniqely defined by
\[
\langle \nabla u(z), h\rangle_H = \frac{\partial u}{\partial h} (z), \; h \in H.
\]
Then 
\[
\mathcal{E}_m (u,v) = \frac{1}{2}\int
\langle \nabla u, \nabla v\rangle_H dm.
\]

Let $A$ be a map from $E$ to the space of all bounded linear operators 
denoted by $\mathcal{L}(H)$, such that $z \mapsto \langle A(z)h_1, h_2\rangle_H$ is $\mathcal{B}(E)$-measurable for all $h_1$, $h_2 \in H$. 
Additionally, suppose that there exists $C > 0$ with
\[
\langle A(z)h, h \rangle_H \geq C \|h\|^2_H \; \mbox{ for all} \; h \in H,
\]
and that $\|\widetilde{A}\| \in L^1(E, m)$ and $\|\check{A}\| \in L^\infty(E, m)$, 
where $\widetilde{A} := \frac{1}{2}(A + \hat{A})$, $\check{A} := \frac{1}{2}(A - \hat{A})$ and $\hat{A}(z)$ denotes the adjoint of $A(z)$, $z\in E$.
Let $b, c \in L^\infty(E \rightarrow H, m)$ such that for all $u \in \mathcal{F}C_b^\infty$ it holds
\begin{equation} \label{eq 5.2}
\int \langle b, \nabla u\rangle_H dm, \; \; \; \int \langle c, \nabla u\rangle_Hdm \geq 0. 
\end{equation}
Then by \cite{MaRo92}, Chapter II, Section 3, 
\[
\mathcal{E}'(u, v) := \int \langle A \nabla u, \nabla v\rangle_H dm + \int u \langle b, \nabla v \rangle_H dm + \int \langle c, \nabla u\rangle_H v dm
\]
is a closable densely defined positive bilinear form on $L^2(E, m)$, whose closure is a Dirichlet form.
Moreover, if $J$ is a symmetric finite positive measure on $(E \times E, \mathcal{B} \otimes \mathcal{B})$ such that $(\mathcal{E}_J, \mathcal{F}C_b^\infty)$ given by
\[
\mathcal{E}_J(u,v):=\int \int (u(x) - u(y)) (v(x) - v(y)) J(dxdy),  \; u,v \in \mathcal{F}C_b^\infty, 
\]
is closable, then $(\mathcal{E} := \mathcal{E}' + \mathcal{E}_J, \mathcal{F}C_b^\infty)$ is closable and its closure is a Dirichlet form.
Note that since $1 \in \mathcal{F}C_b^\infty$, $m(E) < \infty$, and $\mathcal{E}(1,1)=0$, by Corollary \ref{coro 5.3} it follows that $\mathcal{E}$, $\mathcal{E}'$, and $\mathcal{E}_J$ are recurrent.

\begin{coro} \label{coro 5.8} 
The following assertions hold.

i) Let $\mathcal{E}_m$ be irreducible (equivalently, let $m$ be extremal in $\mathcal{G}^b$). 
Then $\mathcal{E}$ is irreducible.

\vspace{0.1cm}

ii) Assume that $\int\langle b + c, \nabla u\rangle_Hdm = 0$ for all $u \in \mathcal{F}C_b^\infty$ and there exists $C' > 0$ such that $<Ah,h>_H \leq C' \|h\|_H^2$ for all $h \in H$. 
Then $\mathcal{E}'$ is irreducible if and only if $\mathcal{E}_m$ is irreducible.
\end{coro} 

\begin{proof}
i) By the strict ellipticity of $A$ and the relations (\ref{eq 5.2}), it is straightforward to check that $D(\mathcal{E}) \subset D(\mathcal{E}_m)$ and if $u \in D(\mathcal{E})$ such that $\mathcal{E}(u,u)=0$, then $\mathcal{E}_m(u,u)=0$. 
Now, the assertion follows by Corollary \ref{coro 5.3}, ii).

ii) Notice that if $u \in \mathcal{F}C_b^\infty$, then integrating by parts and using the first assumption in ii) we get 
\[
\mathcal{E}'(u,u) =  \int\langle A \nabla u, \nabla u\rangle_H dm.
\]
Therefore, $C \mathcal{E}(u,u) \leq \mathcal{E}'(u,u) \leq C' \mathcal{E}(u,u)$, $D(\mathcal{E}') = D(\mathcal{E})$, and the statement follows by applying Corollary \ref{coro 5.3}, ii).
\end{proof}

\begin{acknowledgements}\label{ackref}
We would like to thank Nicu Boboc for fruitful discussions in the initial phase of this paper. 
\end{acknowledgements}

\affiliationone{
Lucian Beznea\\
Simion Stoilow Institute of Mathematics\\ 
of the Romanian Academy,\\
Research unit No. 2, P.O. Box  1-764,\\
RO-014700 Bucharest, Romania,\\
and University of Bucharest, Faculty\\ of Mathematics and Computer
Science
\email{e-mail: lucian.beznea@imar.ro}}
\affiliationtwo{
Iulian C\^impean\\
Simion Stoilow Institute of Mathematics\\
of the Romanian Academy,\\
Research unit No. 2, P.O. Box  1-764,\\
RO-014700 Bucharest\\
Romania
\email{iulian.cimpean@imar.ro}}
\affiliationthree{
Michael R\"ockner\\
Fakult\"at f\"ur Mathematik, Universit\"at\\
Bielefeld,\\
Postfach 100 131, D-33501 Bielefeld\\
Germany
\email{roeckner@mathematik.uni-bielefeld.de}
}
\end{document}